\documentclass[final,leqno]{siamltex}

\usepackage{amsmath,amssymb,amscd,amsxtra,amsfonts,mathrsfs}
\usepackage{epsf,graphicx,epsfig,color,latexsym,cite,cases}

\title{inverse obstacle scattering for elastic waves in three dimensions}

\author{Peijun Li\thanks{Department of Mathematics, Purdue
University, West Lafayette, Indiana 47907, USA. This author's research was
supported in part by the NSF grant DMS-1151308.
({\tt lipeijun@math.purdue.edu})}  \and  Xiaokai Yuan\thanks{Department of
Mathematics, Purdue University, West Lafayette, Indiana 47907, USA. ({\tt
yuan170@math.purdue.edu})} }

\begin{document}

\maketitle

\begin{abstract}
Consider an exterior problem of the three-dimensional elastic wave equation,
which models the scattering of a time-harmonic plane wave by a rigid obstacle.
The scattering problem is reformulated into a boundary value problem
by introducing a transparent boundary condition. Given the incident field, the
direct problem is to determine the displacement of the wave field from the known
obstacle; the inverse problem is to determine the obstacle's surface from the
measurement of the displacement on an artificial boundary enclosing the
obstacle. In this paper, we consider both the direct and inverse problems. The
direct problem is shown to have a unique weak solution by examining its
variational formulation. The domain derivative is studied and a frequency
continuation method is developed for the inverse problem. Numerical experiments
are presented to demonstrate the effectiveness of the proposed method.
\end{abstract}



\pagestyle{myheadings}
\thispagestyle{plain}
\markboth{P. Li and X.Yuan}{Inverse Obstacle Scattering For Elastic Waves}

\section{Introduction}

The obstacle scattering problem, which concerns the scattering of a
time-harmonic incident wave by an impenetrable medium, is a fundamental problem
in scattering theory \cite{ck-83}. It has played an important role in many
scientific areas such as geophysical exploration, nondestructive testing, radar
and sonar, and medical imaging. Given the incident field, the direct obstacle
scattering problem is to determine the wave field from the known obstacle; the
inverse obstacle scattering problem is to determine the shape of the obstacle
from the measurement of the wave field. Due to the wide applications and rich
mathematics, the direct and inverse obstacle scattering problems have been
extensively studied for acoustic and electromagnetic waves by numerous
researchers in both the engineering and mathematical communities \cite{ck-98,
m-03, n-00}.

Recently, the scattering problems for elastic waves have received
ever-increasing attention because of the significant applications in geophysics
and seismology \cite{amg-15, bc-ip05, ll-86}. The propagation of elastic waves
is governed by the Navier equation, which is complex due to the coupling of the
compressional and shear waves with different wavenumbers. The inverse elastic
obstacle scattering problem is investigated mathematically in \cite{cgk-ip01,
ey-ip10, hh-ip93} for the uniqueness and numerically in \cite{hlls-sjis14,
l-ip15} for the shape reconstruction. We refer to for some more related direct
and inverse scattering problems for elastic waves \cite{aa-siap01, bhsy,
hks-ip13, hlz-ip13, k-ip96, ks-je15, lw-ip15, lwz-ip15, lwz-jcp16, lwz-aa16,
nt-siap96, nu-siam95, nu-im94, t-siap15}.

In this paper, we consider the direct and inverse obstacle scattering problems
for elastic waves in three dimensions. The goal is fourfold: (1) develop a
transparent boundary condition to reduce the scattering problem into a boundary
value problem; (2) establish the well-posedness of the solution for the direct
problem by studying its variational formulation; (3) characterize the domain
derivative of the wave field with respect to the variation of the obstacle's
surface; (4) propose a frequency continuation method to reconstruct the
obstacle's surface. This paper significantly extends the two-dimensional work
\cite{lwwz-ip16}. We need to consider more complicated Maxwell's equation and
associated spherical harmonics when studying the transparent boundary condition.
Computationally, it is also more intensive.

The rigid obstacle is assumed to be embedded in an open space filled with a
homogeneous and isotropic elastic medium. The scattering problem is reduced
into a boundary value problem by introducing a transparent boundary condition on
a sphere. We show that the direct problem has a unique weak solution by
examining its variational formulation. The proofs are based on asymptotic
analysis of the boundary operators, the Helmholtz decomposition, and the
Fredholm alternative theorem.

The calculation of domain derivatives, which characterize the variation of the
wave field with respect to the perturbation of the boundary of an medium, is an
essential step for inverse scattering problems. The domain derivatives have been
discussed by many authors for the inverse acoustic and electromagnetic obstacle
scattering problems \cite{hk-siap04, k-ip93, p-mmas96}. Recently, the domain
derivative is studied in \cite{l-siap12} for the elastic wave by using boundary
integral equations. Here we present a variational approach to show that it is
the unique weak solution of some boundary value problem. We propose a
frequency continuation method to solve the inverse problem. The method requires
multi-frequency data and proceed with respect to the frequency. At each
frequency, we apply the descent method with the starting point given by
the output from the previous step, and create an approximation to the surface
filtered at a higher frequency. Numerical experiments are presented to
demonstrate the effectiveness of the proposed method. A topic review can be
found in \cite{bllt_ip15} for solving inverse scattering problems with
multi-frequencies to increase the resolution and stability of reconstructions.

The paper is organized as follows. Section 2 introduces the formulation of
the obstacle scattering problem for elastic waves. The direct problem is
discussed in section 3 where well-posedness of the solution is
established. Section 4 is devoted to the inverse problem. The domain derivative
is studied and a frequency continuation method is introduced for the
inverse problem. Numerical experiments are presented in section 5. The paper is
concluded in section 6. To avoid distraction from the main
results, we collect in the appendices some necessary notation and useful results
on the spherical harmonics, functional spaces, and transparent boundary
conditions.

\section{Problem formulation}

Consider a bounded and rigid obstacle $D\subset\mathbb R^3$ with a Lipschitz
boundary $\partial D$. The exterior domain $\mathbb{R}^3\setminus\bar{D}$ is
assumed to be filled with a homogeneous and isotropic elastic medium, which has
a unit mass density and constant Lam\'{e} parameters $\lambda, \mu$ satisfying
$\mu>0, \lambda+\mu>0$. Let $B_R=\{\boldsymbol{x}\in\mathbb{R}^3:\,
|\boldsymbol{x}|< R\}$, where the radius $R$ is large enough such that
$\bar{D}\subset B_R$. Define $\Gamma_R=\{\boldsymbol{x}\in\mathbb{R}^3:\,
|\boldsymbol{x}|=R\}$ and $\Omega=B_R\setminus\bar{D}$.

Let the obstacle be illuminated by a time-harmonic plane wave
\begin{equation}\label{pw}
 \boldsymbol{u}^{\rm inc}=\boldsymbol{d} e^{{\rm
 i}\kappa_{\rm p}\boldsymbol{x}\cdot\boldsymbol{d}}\quad\text{or}\quad   
\boldsymbol{u}^{\rm inc}=\boldsymbol{d}^\perp e^{{\rm
    i}\kappa_{\rm s}\boldsymbol{x}\cdot\boldsymbol{d}},
\end{equation}
where $\boldsymbol{d}$ and $\boldsymbol d^\perp$ are orthonormal vectors,
$\kappa_{\rm p}=\omega/\sqrt{\lambda+2\mu}$ and $\kappa_{\rm
s}=\omega/\sqrt{\mu}$ are the compressional wavenmumber and the shear
wavenumber. Here $\omega>0$ is the angular frequency. It is easy to verify
that the plane incident wave \eqref{pw} satisfies
\begin{equation}\label{uine}
 \mu\Delta\boldsymbol{u}^{\rm inc}+(\lambda+\mu)\nabla
 \nabla\cdot\boldsymbol{u}^{\rm inc}+\omega^2\boldsymbol{u}^{\rm
 inc}=0\quad\text{in}~ \mathbb{R}^3\setminus\bar{D}.
\end{equation}
Let $\boldsymbol u$ be the displacement of the total wave field which 
also satisfies
\begin{equation}\label{une}
 \mu\Delta\boldsymbol{u}+(\lambda+\mu)\nabla\nabla\cdot\boldsymbol{u}
 +\omega^2\boldsymbol{u}=0\quad\text{in}~ \mathbb{R}^3\setminus\bar{D}.
\end{equation}
Since the obstacle is elastically rigid, we have 
\begin{equation}\label{ubc}
 \boldsymbol{u}=0\quad\text{on}~ \partial D.
\end{equation}
The total field $\boldsymbol{u}$ consists of the incident field
$\boldsymbol{u}^{\rm inc}$ and the scattered field $\boldsymbol{v}$:
\[
\boldsymbol{u}=\boldsymbol{u}^{\rm inc}+\boldsymbol{v}.
\]
Subtracting \eqref{uine} from \eqref{une} yields that $\boldsymbol v$
satisfies
\begin{equation}\label{vne}
  \mu\Delta\boldsymbol{v}+(\lambda+\mu)\nabla\nabla\cdot\boldsymbol{v}
 +\omega^2\boldsymbol{v}=0\quad\text{in}~ \mathbb{R}^3\setminus\bar{D}.
\end{equation}

For any solution $\boldsymbol{v}$ of \eqref{vne}, we introduce the
Helmholtz decomposition by using a scalar function $\phi$ and a divergence
free vector function $\boldsymbol\psi$:
\begin{equation}\label{hdv}
 \boldsymbol{v}=\nabla\phi+\nabla\times\boldsymbol{\psi}, \quad 
\nabla\cdot\boldsymbol\psi=0.
\end{equation}
Substituting \eqref{hdv} into \eqref{vne}, we may verify that  
$\phi$ and $\boldsymbol{\psi}$ satisfy 
\begin{equation}\label{he}
\Delta\phi+\kappa^2_{\rm p}\phi=0,\quad \Delta\boldsymbol\psi+\kappa^2_{\rm
s}\boldsymbol\psi=0.
\end{equation}
In addition, we require that $\phi$ and $\boldsymbol{\psi}$ satisfy the
Sommerfeld radiation condition:
\begin{equation}\label{ksrc}
\lim_{r\to\infty}r\left(\partial_r\phi -{\rm
    i}\kappa_{\rm p}\phi\right)=0, \quad
 \lim_{r\to\infty}r\left(\partial_r\boldsymbol{\psi} -{\rm
    i}\kappa_{\rm s}\boldsymbol{\psi}\right)=0, \quad r=|\boldsymbol{x}|.
\end{equation}

Using the identity
$\nabla\times(\nabla\times\boldsymbol\psi)=-\Delta\boldsymbol\psi+\nabla (\nabla
\cdot\boldsymbol\psi),$ we have from \eqref{he} that $\boldsymbol\psi$ satisfies
the Maxwell equation:
\begin{equation}\label{ms}
 \nabla\times(\nabla\times\boldsymbol\psi)-\kappa^2_{\rm s}\boldsymbol\psi=0. 
\end{equation}
It can be shown (cf. \cite[Theorem 6.8]{ck-98}) that the Sommerfeld radiation
for $\boldsymbol\psi$ in \eqref{ksrc} is equivalent to the Silver--M\"{u}ller
radiation condition:
\begin{equation}\label{smrc}
 \lim_{r\to\infty}\left( (\nabla\times\boldsymbol\psi)\times\boldsymbol
x-{\rm i}\kappa_{\rm s}r\boldsymbol\psi\right)=0,\quad r=|\boldsymbol x|. 
\end{equation}

Given $\boldsymbol u^{\rm inc}$, the direct problem is to determine $\boldsymbol
u$ for the known obstacle $D$; the inverse problem is to determine the
obstacle's surface $\partial D$ from the boundary measurement of $\boldsymbol u$
on $\Gamma_R$. Hereafter, we take the notation of $a\lesssim b$ or $a\gtrsim b$
to stand for $a\leq Cb$ or $a\geq Cb$, where $C$ is a positive constant whose
specific value is not required but should be clear from the context.

\section{Direct scattering problem}

In this section, we study the variational formulation for the direct problem
and show that it admits a unique weak solution.

\subsection{Transparent boundary condition}

We derive a transparent boundary condition on $\Gamma_R$. Given $\boldsymbol
v\in\boldsymbol L^2(\Gamma_R)$, it has the Fourier expansion:
\[
 \boldsymbol v(R, \theta,
\varphi)=\sum_{n=0}^{\infty}\sum_{m=-n}^{n}v_{1n}^m\boldsymbol{T}_{n}^{m}
(\theta, \varphi)+ v_{2n}^m\boldsymbol{V}_{n}^{m}(\theta,
\varphi)+v_{3n}^m\boldsymbol{W}_{n}^{m}(\theta, \varphi),
\]
where $\{(\boldsymbol T_n^m, \boldsymbol V_n^m, \boldsymbol W_n^m): n=0,1,
\dots, m=-n,\dots, n\}$ is an orthonormal system in $\boldsymbol
L^2(\Gamma_R)$ and $v_{jn}^m$ are the Fourier coefficients of $\boldsymbol v$
on $\Gamma_R$. Define a boundary operator
\begin{equation}\label{bo}
 \mathscr{B}\boldsymbol{v}=\mu\partial_r\boldsymbol{v}
+(\lambda+\mu)(\nabla\cdot\boldsymbol{v})\boldsymbol{e}_r\quad\text{on}
~\Gamma_R,
\end{equation}
which is assumed to have the Fourier expansion:
\begin{equation}\label{boe1}
(\mathscr{B}\boldsymbol{v})(R, \theta,
\varphi)=\sum_{n=0}^{\infty}\sum_{m=-n}^{n}
w_{1n}^m\boldsymbol{T}_{n}^{m}(\theta,
\varphi)+w_{2n}^m\boldsymbol{V}_{n}^{m}(\theta,
\varphi)+w_{3n}^m\boldsymbol{W}_{n}^{m}(\theta, \varphi).
\end{equation}

Taking $\partial_{r}$ of $\boldsymbol v$ in \eqref{D-v}, evaluating it at
$r=R$, and using the spherical Bessel differential equations \cite{w-22}, we
get 
\begin{align}\label{pRv}
 \partial_{r}&\boldsymbol{v}(R,
\theta, \varphi)=\sum_{n=0}^{\infty}\sum_{m=-n}^{n}\Bigg[\frac{\sqrt {
n(n+1)}\phi_{n}^{m}}{R^2} (z_{n}(\kappa_{\rm p}R)-1)-
\frac{\psi_{2n}^m}{R^2}\Big(1+z_{n}(\kappa_{\rm s}R)\notag\\
&+(R\kappa_{\rm s})^2-n(n+1)\Big)\Bigg]
\boldsymbol{T}_{n}^{m}+\Bigg[\frac{\kappa_{\rm s}^2
\psi_{3n}^m}{\sqrt{n(n+1)}}
z_{n}(\kappa_{\rm s}R)\Bigg]\boldsymbol{V}_{n}^{m}
+\Bigg[\frac{\phi_{n}^{m}}{R^2}\big(n(n+1)\notag\\
&-(R\kappa_{\rm p})^2-2z_{n}(\kappa_{\rm p} R)\big)
+\frac{\sqrt{n(n+1)} \psi_{2n}^m}{R^2}(z_{n}(\kappa_{\rm
s}R)-1)\Bigg] \boldsymbol {W}_{n}^m,
\end{align}
where $z_n(t)=t h_n^{(1)'}(t)/h_n^{(1)}(t), h_n^{(1)}$ is the spherical Hankel
function of the first kind with order $n$, $\phi_n^m$ and $\psi_{jn}^m$ are the
Fourier coefficients for $\phi$ and $\boldsymbol\psi$ on $\Gamma_R$,
respectively.

Noting \eqref{D-v} and using
$\nabla\cdot\boldsymbol{v}=\Delta\phi=\frac{2}{r}\partial_{r}\phi+\partial_ {r
}^2\phi +\frac{1}{r}\Delta_{\Gamma_R}\phi,$ we have
\begin{align}\label{dv}
 \nabla\cdot\boldsymbol{v}(r,
\theta, \varphi)=\sum_{n=0}^{\infty}\sum_{m=-n}^{n}\frac{\phi_{n}^{m}}
 {h_{n}^{(1)}(\kappa_{\rm p}R)}\Bigg[\frac{2}{r}\frac{\rm d}{{\rm
d}r}h_{n}^{(1)}(\kappa_{\rm p} r) +\frac{{\rm d}^2}{{\rm
d}r^2}h_{n}^{(1)}(\kappa_{\rm p}r)\notag\\
-\frac{n(n+1)}{r^2}h_{n}^{(1)}
(\kappa_{\rm p}r)\Bigg] X_{n}^{m},
\end{align}
where $\Delta_{\Gamma_R}$ is the Laplace--Beltrami operator on
$\Gamma_R$.

Combining \eqref{bo} and \eqref{pRv}--\eqref{dv}, we obtain 
\begin{align}\label{boe2}
 \mathscr{B}&\boldsymbol{v}=\sum_{n=0}^{\infty}\sum_{m=-n}^{n}
\frac{\mu}{R^2}\Big[\sqrt{n(n+1)} (z_n(\kappa_{\rm p}
R)-1) \phi_n^m-\big(1+z_n(\kappa_{\rm s}R)+(R\kappa_{\rm
s})^2\notag\\
&-n(n+1)\big)\psi_{2n}^m\Big]\boldsymbol{T}_n^m
+\frac{\mu
\kappa_{\rm s}^2}{\sqrt{n(n+1)}}z_n (\kappa_{\rm s}
R)\psi_{3n}^m\boldsymbol{V}_n^m
+ \frac{1}{R^2}\Big[ \mu\big(n(n+1)-(R\kappa_{\rm
p})^2\notag\\
&-2z_n(\kappa_{\rm p} R)\big) \phi_n^m
 +\mu \sqrt{n(n+1)} (z_n(\kappa_{\rm s}
R)-1)\psi_{2n}^m-(\lambda+\mu)(\kappa_{\rm p}
R)^2\phi_n^m\Big]\boldsymbol{W}_n^m.
\end{align}
Comparing \eqref{boe1} with \eqref{boe2}, we have
\begin{equation}\label{wtp}
(w_{1n}^m, w_{2n}^m, w_{3n}^m)^\top
=\frac{1}{R^2}G_n (\phi^{m}_n, \psi_{2n}^m, \psi_{3n}^m)^\top,
\end{equation}
where the matrix
\[
 G_n=\begin{bmatrix}
      0 & 0 & G_{13}^{(n)}\\
      G_{21}^{(n)}& G_{22}^{(n)}&0\\
      G_{31}^{(n)}& G_{32}^{(n)} & 0
     \end{bmatrix}.
\]
Here 
\begin{align*}
 G_{13}^{(n)}&=\frac{\mu(\kappa_{\rm s}R)^{2}z_{n}(\kappa_{\rm
s}R)}{\sqrt{n(n+1)}},\quad
G_{21}^{(n)}=\mu \sqrt{n(n+1)} (z_{n}(\kappa_{\rm p}R)-1),\\
G_{22}^{(n)}&=\mu\left(n(n+1)-(\kappa_{\rm s}R)^2-1-z_{n}
    (\kappa_{\rm s}R)\right),\\
G_{31}^{(n)}&=\mu\left(   n(n+1)-(\kappa_{\rm p}R)^2-2z_{n}(\kappa_{\rm
p}R)\right)-(\lambda+\mu)(\kappa_{\rm p} R)^2,\\
G_{32}^{(n)}&=\mu\sqrt{n(n+1)}(z_{n}(\kappa_{\rm s}R)-1).
\end{align*}

Let $\boldsymbol v_n^m=(v_{1n}^m, v_{2n}^m, v_{3n}^m)^\top, \quad M_n\boldsymbol
v_n^m=\boldsymbol b_n^m=(b_{1n}^m, b_{2n}^m, b_{3n}^m)^\top$,
where the matrix
\[
 M_n=\begin{bmatrix}
          M^{(n)}_{11} &  0 & 0\\
          0 &  M^{(n)}_{22} & M^{(n)}_{23}\\
          0&  M^{(n)}_{32} & M^{(n)}_{33}
         \end{bmatrix}.
\]
Here
\begin{align*}
M^{(n)}_{11}&=\left(\frac{\mu }{R}\right) z_{n}(\kappa_{\rm
s}R),\quad
M^{(n)}_{22}=-\left(\frac{\mu}{R}\right)\left(1+\frac{(\kappa_{\rm
s}R)^2 z_{n}(\kappa_{\rm p}R)}{\Lambda_n}\right),\\[5pt]
M^{(n)}_{23}&=\sqrt{n(n+1)}\left(\frac{\mu}{R}\right)\left(1+\frac{(\kappa_{
\rm s}R)^2}{\Lambda_n}\right),\\[5pt]
M^{(n)}_{32}&=\sqrt{n(n+1)}\left(\frac{\mu}{R}+\frac{(\lambda+2\mu)}{R}\frac{
(\kappa_ { \rm p}R)^2}{\Lambda_n}\right),\\[5pt]
M^{(n)}_{33}&=-\frac{(\lambda+2\mu)}{R}\frac{(\kappa_{\rm
p}R)^2}{\Lambda_n} (1+z_n(\kappa_{\rm s}R))-2\left(\frac{\mu}{R}\right),
\end{align*}
where $\Lambda_n =z_n(\kappa_{\rm p}R)(1+z_n(\kappa_{\rm s} R))-n(n+1)$.

Using the above notation and combining \eqref{wtp} and \eqref{ptv}, we derive
the transparent boundary condition: 
\begin{equation}\label{tbc}
 \mathscr{B}\boldsymbol{v}=\mathscr{T}\boldsymbol{v}:=\sum_{n=0}^{\infty}\sum_{
m=-n}^{n} b_{1n}^m\boldsymbol T_n^m +b_{2n}^m\boldsymbol V_n^m
+b_{3n}^m \boldsymbol W_n^m\quad\text{on} ~
\Gamma_R.
\end{equation}

\begin{lemma}\label{mpd}
 The matrix $\hat{M}_n=-\frac{1}{2}(M_n+M_n^*)$ is positive
definite for sufficiently large $n$.
\end{lemma}

\begin{proof}
Using the asymptotic expansions of the spherical Bessel functions \cite{w-22},
we may verify that 
\begin{align*}
z_n(t)&=-(n+1)+\frac{1}{16n}t^4+\frac{1}{2n}t^2+O\left(\frac{1}{n^2}\right),\\
\Lambda_n(t)&=-\frac{1}{16}(\kappa_{\rm p}t)^4-\frac{1}{16}(\kappa_{\rm s}t)^4
-\frac{1}{2}(\kappa_{\rm p}t)^2-\frac{1}{2}(\kappa_{\rm s}t)^2
+O\left(\frac{1}{n}\right).
\end{align*}
It follows from straightforward calculations that   
 \[
  \hat{M}_n=\begin{bmatrix}
             \hat{M}^{(n)}_{11} & 0& 0\\
             0 & \hat{M}^{(n)}_{22} & \hat{M}^{(n)}_{23}\\
             0 & \hat{M}^{(n)}_{32} & \hat{M}^{(n)}_{33}
            \end{bmatrix},
 \]
 where
 \begin{align*}
\hat{M}^{(n)}_{11} &
=\left(\frac{\mu}{R}\right)(n+1)+O\left(\frac{1}{n}\right),\quad
\hat{M}^{(n)}_{22} = -\left(\frac{\omega^2
R}{\Lambda_n}\right)(n+1)+O\left(1\right),\\
\hat{M}^{(n)}_{23} &=
-\left(\frac{\mu}{R}+\frac{\omega^2
R}{\Lambda_n}\right) \sqrt{n(n+1)}+O(1),\\
\hat{M}^{(n)}_{32}
&=-\left(\frac{\mu}{R}+\frac{\omega^2
R}{\Lambda_n}\right) \sqrt{n(n+1)}+O(1),\\
\hat{M}^{(n)}_{33}
&=\frac{2\mu}{R}+\frac{\omega^2R}{\Lambda_n}(1+z_n(\kappa_{\rm s} R))=
-\left(\frac{\omega^2 R}{\Lambda_n}\right)n+O(1).
\end{align*}
For sufficiently large $n$, we have 
\[
 \hat{M}^{(n)}_{11}>0\quad\text{and}\quad \hat{M}^{(n)}_{22}>0,
\]
which gives
\[
{\rm
det}[(\hat{M}_n)_{(1:2,1:2)}]=\hat{M}^{(n)}_{11}\hat{M}^{(n)}_{22}>0. 
\]
Since $\Lambda_n<0$ for sufficiently large $n$, we have 
\[
\hat{M}^{(n)}_{22}
 \hat{M}^{(n)}_{33}-\left(\hat{M}^{(n)}_{23}\right)^2=n(n+1)\left[
\left(\frac{\omega^2 R}{\Lambda_n}\right)^2-\left(\frac{\mu}{R}+\frac{\omega^2
R}{\Lambda_n}\right)^2\right]+O(n)>0.
\]
A simple calculation yields 
\[
 {\rm det}[\hat{M}_{n}]=\hat{M}^{(n)}_{11}\left(\hat{M}^{(n)}_{22}
 \hat{M}^{(n)}_{33}-\left(\hat{M}^{(n)}_{23}\right)^2\right)>0,
\]
which completes the proof by applying Sylvester's criterion.
\end{proof}

\begin{lemma}\label{tbo}
The boundary operator $\mathscr{T}: \boldsymbol H^{1/2}(\Gamma_R)\to
\boldsymbol H^{-1/2}(\Gamma_R)$ is continuous, i.e.,
\[
\| \mathscr{T}\boldsymbol{u}\|_{\boldsymbol H^{-1/2}(\Gamma_R)} \lesssim \|
\boldsymbol{u}\|_{\boldsymbol H^{1/2}(\Gamma_R)},\quad\forall\,\boldsymbol{u}\in
\boldsymbol H^{1/2}(\Gamma_R).
  \]
\end{lemma}

\begin{proof}
For any given $\boldsymbol u\in \boldsymbol H^{1/2}(\Gamma_R)$, it has the
Fourier expansion
\[
 \boldsymbol u(R, \theta, \varphi)=\sum_{n=0}^\infty\sum_{m=-n}^n u_{1n}^m
\boldsymbol T_n^m(\theta, \varphi) +u_{2n}^m\boldsymbol V_n^m(\theta,
\varphi)+u_{3n}^m \boldsymbol W_n^m(\theta, \varphi). 
\]
Let $\boldsymbol u_n^m=(u_{1n}^m, u_{2n}^m, u_{3n}^m)^\top$. It follows from 
\eqref{tbc} and the asymptotic expansions of $M_{ij}^{(n)}$ that
\begin{align*}
\| \mathscr{T} \boldsymbol{u} \|^2_{\boldsymbol H^{1/2}(\Gamma_R)} &=
\sum_{n=0}^{\infty}\sum_{m=-n}^{n}
\left( 1 + n(n+1) \right)^{-1/2} |M_n \boldsymbol{u}^{m}_{n}|^2 \\
&\lesssim \sum_{n=0}^{\infty}\sum_{m=-n}^{n}\left(1 +
n(n+1)\right)^{1/2}|\boldsymbol{u}^{m}_{n}|^2 = \| \boldsymbol{u}
\|^2_{\boldsymbol H^{1/2}(\Gamma_R)},
\end{align*}
which completes the proof.
\end{proof}

\subsection{Uniqueness}

It follows from the Dirichlet boundary condition \eqref{ubc} and the Helmholtz
decomposition \eqref{hdv} that
\begin{equation}\label{bc}
 \boldsymbol{v}=\nabla\phi+ \nabla\times\boldsymbol{\psi}=-\boldsymbol{u}^{\rm
inc}\quad\text{on} ~ \partial D.
\end{equation}
Taking the dot product and the cross product of \eqref{bc} with the unit normal
vector $\boldsymbol\nu$ on $\partial D$, respectively, we get 
\[
 \partial_{\boldsymbol\nu}\phi
+(\nabla\times\boldsymbol\psi)\cdot\boldsymbol\nu=-u_1,\quad
(\nabla\times\boldsymbol\psi)\times\boldsymbol\nu
+\nabla\phi\times\boldsymbol\nu=-\boldsymbol u_2,
\]
where
\[
 u_1=\boldsymbol u^{\rm inc}\cdot\boldsymbol\nu,\quad \boldsymbol
u_2=\boldsymbol u^{\rm inc}\times\boldsymbol\nu.
\]
We obtain a coupled boundary value problem for the potential functions $\phi$
and
$\boldsymbol\psi$:
\begin{align}\label{cbvp}
  \begin{cases}
    \Delta\phi + \kappa_{\rm p}^2\phi = 0 ,\quad
    \nabla\times(\nabla\times\boldsymbol{\psi}) -
\kappa_{\rm s}^2\boldsymbol{\psi} = 0 ,
&\quad {\rm in}~\Omega,\\
\partial_{\boldsymbol\nu}
\phi+(\nabla\times\boldsymbol\psi)\cdot\boldsymbol\nu=-u_ 1, \quad
(\nabla\times\boldsymbol\psi)\times\boldsymbol\nu+\nabla\phi\times\boldsymbol\nu
=-\boldsymbol u_2 &\quad {\rm on}~\partial D,\\
    \partial_r \phi- \mathscr{T}_1\phi =0 ,\quad
    (\nabla\times \boldsymbol{\psi})\times \boldsymbol e_r-
{\rm i}\kappa_{\rm s}\mathscr{T}_2\boldsymbol{\psi}_{\Gamma_R} =0 &\quad {\rm
on} ~ \Gamma_R.
  \end{cases}
\end{align}
where $\mathscr T_1$ and $\mathscr T_2$ are the transparent boundary operators
given in \eqref{tphi} and \eqref{tpsi}, respectively.

Multiplying test functions $(p, \boldsymbol q)\in H^1(\Omega)\times\boldsymbol
H({\rm curl}, \Omega)$,  we arrive at the weak formulation of
\eqref{cbvp}: To find $(\phi, \boldsymbol{\psi})\in H^1(\Omega)\times
\boldsymbol H({\rm curl}, \Omega)$ such that
\begin{equation}\label{cvp}
 a(\phi, \boldsymbol{\psi}; p, \boldsymbol q)=\langle u_1,
p\rangle_{\partial D}+\langle \boldsymbol{u_2},
\boldsymbol q\rangle_{\partial
D},\quad\forall\, (p, \boldsymbol q)\in H^1(\Omega)\times\boldsymbol
H({\rm curl}, \Omega),
\end{equation}
where the sesquilinear form
\begin{align*}
 a(\phi, \boldsymbol{\psi}; p, \boldsymbol q)=(\nabla\phi,
\nabla p)+ (\nabla\times\boldsymbol{\psi},
\nabla\times\boldsymbol{q})-\kappa^2_{\rm p}(\phi,
p)-\kappa^2_{\rm s}(\boldsymbol{\psi},
\boldsymbol{q})-\langle(\nabla\times\boldsymbol{\psi})\cdot\boldsymbol\nu,
p\rangle_{\partial D}\notag\\
-\langle\nabla\phi\times\boldsymbol\nu,
\boldsymbol{q}\rangle_{\partial
D}-\langle\mathscr{T}_1\phi,
p\rangle_{\Gamma_R}-{\rm
i}\kappa_{\rm s}\langle\mathscr{T}_2\boldsymbol{\psi}_{\Gamma_R},
\boldsymbol{q}_{\Gamma_R}\rangle_{\Gamma_R}.
\end{align*}

\begin{theorem}\label{uni}
 The variational problem \eqref{cvp} has at most one solution.
\end{theorem}

\begin{proof}
It suffices to show that $\phi=0, \boldsymbol{\psi}=0$ in $\Omega$ if
$u_1=0, \boldsymbol{u_2}=0$ on $\partial D$. If
$(\phi, \boldsymbol{\psi})$ satisfy the homogeneous variational problem
\eqref{cvp}, then
we have
\begin{align}\label{uni_1}
(\nabla\phi, \nabla\phi)+(\nabla\times\boldsymbol{\psi},
\nabla\times\boldsymbol{\psi})-\kappa^2_{\rm p}(\phi,
\phi)-\kappa^2_{\rm s}(\boldsymbol{\psi}, \boldsymbol{\psi})
-\langle (\nabla\times\boldsymbol{\psi})\cdot\boldsymbol\nu,
\phi\rangle_{\partial D}\notag\\
-\langle \nabla\phi\times\boldsymbol\nu,
\boldsymbol{\psi}\rangle_{\partial D}-\langle\mathscr{T}_1\phi,
\phi\rangle_{\Gamma_R}-{\rm
i}\kappa_{\rm s}\langle\mathscr{T}_2\boldsymbol{\psi}_{\Gamma_R},
\boldsymbol{\psi}_{\Gamma_R}\rangle_{\Gamma_R}=0.
\end{align}
Using the integration by parts, we may verify that 
\[
\langle(\nabla\times\boldsymbol{\psi})\cdot\boldsymbol\nu, \phi\rangle_{\partial
D}=-\langle\boldsymbol{\psi}, \boldsymbol\nu\times\nabla\phi\rangle_{\partial
D}=\langle\boldsymbol{\psi}, \nabla\phi\times\boldsymbol\nu\rangle_{\partial
D},
\]
which gives 
\begin{equation}\label{uni_2}
\langle(\nabla\times\boldsymbol{\psi})\cdot\boldsymbol\nu, \phi\rangle_{\partial
D}+\langle\nabla\phi\times\boldsymbol\nu, \boldsymbol{\psi}\rangle_{\partial
D}=2{\rm Re}\langle\nabla\phi\times\boldsymbol\nu,
\boldsymbol{\psi}\rangle_{\partial D}.
\end{equation}
Taking the imaginary part of \eqref{uni_1} and using \eqref{uni_2}, we obtain
\[
 {\rm Im}\langle\mathscr{T}_1\phi, \phi\rangle_{\Gamma_R}+\kappa_{\rm s}{\rm
Re}\langle\mathscr{T}_2\boldsymbol{\psi}_{\Gamma_R},
\boldsymbol{\psi}_{\Gamma_R}\rangle_{\Gamma_R}=0,
\]
which gives $\phi=0, \boldsymbol{\psi}=0$ on $\Gamma_R$, due to Lemma \ref{bt1}
and Lemma \ref{bt2}. Using \eqref{tphi} and \eqref{tpsi},  we
have $\partial_r\phi=0, (\nabla\times\boldsymbol{\psi})
\times \boldsymbol e_r=0$ on $\Gamma_R$. By the Holmgren
uniqueness theorem, we have $\phi=0, \boldsymbol{\psi}=0$ in
$\mathbb{R}^3\setminus\bar{B}$. A unique continuation result concludes that
$\phi=0, \boldsymbol{\psi}=0$ in $\Omega$.
\end{proof}

\subsection{Well-posedness}

Using the transparent boundary condition \eqref{tbc}, we obtain a boundary
value problem for $\boldsymbol{u}$:
\begin{align}\label{bvp}
\begin{cases}
\mu\Delta\boldsymbol{u}+(\lambda+\mu)\nabla\nabla\cdot\boldsymbol{u}
+\omega^2\boldsymbol{u}=0 &\quad\text{in} ~ \Omega,\\
\boldsymbol{u}=0 &\quad\text{on} ~ \partial D,\\
\mathscr{B}\boldsymbol{u}=\mathscr{T}\boldsymbol{u}+\boldsymbol{g}
&\quad\text{on} ~\Gamma_R,
\end{cases}
\end{align}
where $\boldsymbol{g}=(\mathscr{B}-\mathscr{T})\boldsymbol{u}^{\rm inc}$. The
variational problem of \eqref{bvp} is to find $\boldsymbol{u}\in
\boldsymbol H^1_{\partial D}(\Omega)$ such that
\begin{equation}\label{vp}
 b(\boldsymbol{u}, \boldsymbol{v})=\langle\boldsymbol{g},
\boldsymbol{v}\rangle_{\Gamma_R},\quad\forall\,
\boldsymbol{v}\in\boldsymbol H^1_{\partial D}(\Omega),
\end{equation}
where the sesquilinear form $b: \boldsymbol H^1_{\partial D}(\Omega)\times 
\boldsymbol H^1_{\partial D}(\Omega)\to\mathbb{C}$ is defined by
\begin{align*}
 b(\boldsymbol{u}, \boldsymbol{v})=\mu \int_\Omega \nabla\boldsymbol{u}:
\nabla\bar{\boldsymbol{v}}\,{\rm d}\boldsymbol{x}+(\lambda+\mu)\int_\Omega
(\nabla\cdot\boldsymbol{u})(\nabla\cdot\bar{\boldsymbol v})\,{\rm
d}\boldsymbol{x}\notag\\
-\omega^2\int_\Omega \boldsymbol{u}\cdot\bar{\boldsymbol{v}}\,{\rm
d}\boldsymbol{x}-\langle\mathscr{T}\boldsymbol{u},
\boldsymbol{v}\rangle_{\Gamma_R}.
\end{align*}
Here $A:B={\rm tr}(A B^\top)$ is the Frobenius inner product of square matrices
$A$ and $B$.

The following result follows from the standard trace theorem of the
Sobolev spaces. The proof is omitted for brevity. 

\begin{lemma}\label{tr}
It holds the estimate
\begin{align*}
\|\boldsymbol{u}\|_{\boldsymbol H^{1/2}(\Gamma_R)}
\lesssim\|\boldsymbol{u}\|_{\boldsymbol H^1(\Omega)}, \quad \forall\,
\boldsymbol{u} \in\boldsymbol H_{\partial
D}^1(\Omega).
\end{align*}
\end{lemma}

\begin{lemma}\label{tr0}
For any $\varepsilon>0$, there exists a positive constant $C(\varepsilon)$ such
that
\begin{align*}
\|\boldsymbol{u}\|_{\boldsymbol L^2(\Gamma_R)} \leq
\varepsilon\|\boldsymbol{u}\|_{\boldsymbol
H^1(\Omega)}+C(\varepsilon)\|\boldsymbol{u}\|_{\boldsymbol L^2(\Omega)}, \quad
\forall\, \boldsymbol{u} \in\boldsymbol H_{\partial D}^1(\Omega).
  \end{align*}
\end{lemma}

\begin{proof}
Let $B'$ be the ball with radius $R'>0$ such that $\bar{B}'\subset D$. Denote
$\tilde\Omega=B \setminus \bar{B}'$. Given $\boldsymbol u\in\boldsymbol
H^1_{\partial D}(\Omega)$, let $\tilde{\boldsymbol u}$ be the
zero extension of $\boldsymbol{u}$ from $\Omega$ to $\tilde{\Omega}$, i.e.,
 \begin{align*}
    \tilde{\boldsymbol u} (\boldsymbol x) =
    \begin{cases}
      \boldsymbol{u} (\boldsymbol x), & \boldsymbol x\in \Omega,\\[5pt]
      0, & \boldsymbol x\in \tilde{\Omega} \setminus \bar{\Omega}.
    \end{cases}
  \end{align*}
The extension of $\tilde{\boldsymbol u}$ has the Fourier expansion
\[
\tilde{\boldsymbol
u}(r,\theta,\varphi)=\sum_{n=0}^{\infty}\sum_{m=-n}^{n}\tilde{u}_{1n}^m
(r)\boldsymbol{T}_{n}^{m}(\theta,\varphi)+
\tilde{u}_{2n}^m(r)\boldsymbol{V}_{n}^{m}(\theta,\varphi)+\tilde{u}_{3n}^m(r)
\boldsymbol{W}_{n}^{m}(\theta,\varphi).
\]
A simple calculation yields 
\[
\|\tilde{\boldsymbol u}\|^2_{\boldsymbol
L^2(\Gamma_R)}=\sum_{n=0}^{\infty}\sum_{
m=-n}^{n} |\tilde{u}_{1n}^{m}(R)|^2+ |\tilde{u}_{2n}^{m}(R)|^2+
|\tilde{u}_{3n}^{m}(R)|^2.
\]
Since $\tilde{\boldsymbol u}(R', \theta, \varphi)=0$, we have
$\tilde{u}_{jn}^m(R')=0$. For any given $\varepsilon>0$, it follows
from Young's
inequality that 
\begin{align*}
|\tilde{u}_{jn}^{m}(R)|^2
&=\int_{R'}^{R}\frac{{\rm d}}{{\rm d}r}|\tilde{u}_{jn}^m(r)|^2
{\rm d}r\leq\int_{R'}^{R}2|\tilde{u}_{jn}^m(r)|\left|\frac{{\rm d}}
{{\rm d}r}\tilde{u}_{jn}^m(r)\right|{\rm d}r\\
&\leq \left(R'\varepsilon\right)^{-2}\int_{R'}^{R}|\tilde{u}_{jn}^m(r)|^2
 {\rm d}r+\left(R'\varepsilon\right)^2\int_{R'}^{R}\left|\frac{{\rm d}}{{\rm
d}r}\tilde{u}_{jn}^m(r)\right|^2{\rm d}r,
\end{align*}
which gives
\[
|\tilde{u}_{jn}^{m}(R)|^2\leq
C(\varepsilon)\int_{R'}^{R}|\tilde{u}_{jn}^m(r)|^2
 r^2{\rm d}r+\varepsilon^2\int_{R'}^{R}\left|\frac{{\rm d}}{{\rm dr}}
 \tilde{u}_{jn}^m(r)\right|^2 r^2{\rm d}r.
\]
The proof is completed by noting that
\[
\|\tilde{\boldsymbol{u}}\|_{\boldsymbol
L^2(\Gamma_R)}=\|\boldsymbol{u}\|_{\boldsymbol L^2(\Gamma_R)},\quad
\|\tilde{\boldsymbol{u}}\|_{\boldsymbol
L^2(\tilde\Omega)}=\|\boldsymbol{u}\|_{\boldsymbol L^2(\Omega)}, \quad
\|\tilde{\boldsymbol{u}}\|_{\boldsymbol
H^1(\tilde\Omega)}=\|\boldsymbol{u}\|_{\boldsymbol H^1(\Omega)}. 
\]
\end{proof}

\begin{lemma} \label{po}
It holds the estimate
\begin{align*}
\|\boldsymbol{u}\|_{\boldsymbol H^1(\Omega)} \lesssim
\|\nabla\boldsymbol{u}\|_{\boldsymbol L^2(\Omega)},
\quad \forall\,\boldsymbol{u}\in\boldsymbol H_{\partial D}^1(\Omega).
\end{align*}
\end{lemma}

\begin{proof}
As is defined in the proof of Lemma \ref{tr0}, let $\tilde{\boldsymbol u}$ be
the zero extension of $\boldsymbol{u}$ from $\Omega$ to $\tilde\Omega$. It
follows from the Cauchy--Schwarz inequality that
\begin{align*}
 |\tilde{\boldsymbol u}(r, \theta, \varphi)|^2 = \left| \int_{R'}^r \partial_r
\tilde{\boldsymbol u}(r, \theta, \varphi){\rm d}r \right|^2 \lesssim
\int_{R'}^R \left|\partial_r \tilde{\boldsymbol u}(r, \theta, \varphi) \right|^2
{\rm d}r.
\end{align*}
Hence we have
\begin{align*}
\|\tilde{\boldsymbol{u}}\|^2_{\boldsymbol L^2(\tilde{\Omega})} &=
\int_{R'}^{R}\int_{0}^{2\pi}
\int_{0}^{\pi}|\tilde{\boldsymbol{u}}(r, \theta, \varphi)|^2
r^2{\rm d}r{\rm d}\theta{\rm d}\varphi\lesssim\int_{R'}^{R}\int_{0}^{2\pi}
\int_{0}^{\pi}\int_{R'}^{R}|\partial_r \tilde{\boldsymbol{u}}(r, \theta,
\varphi)|^2{\rm d}r{\rm d}\theta{\rm d}\varphi{\rm d}r\\
&\lesssim \int_{R'}^{R}\int_{0}^{2\pi}
\int_{0}^{\pi}|\partial_r \tilde{\boldsymbol{u}}(r, \theta, \varphi)|^2{\rm
d}r{\rm d}\theta{\rm d}\varphi\lesssim
\|\nabla\tilde{\boldsymbol{u}}\|^2_{\boldsymbol L^2(\tilde{\Omega})}.
\end{align*}
The proof is completed by noting that
\[
\| \boldsymbol{u} \|_{\boldsymbol L^2(\Omega)}=\| \tilde{\boldsymbol u}
\|_{\boldsymbol L^2(\tilde{\Omega})},\quad \|\nabla
\boldsymbol{u}\|_{\boldsymbol L^2(\Omega)}=\|\nabla
\tilde{\boldsymbol u}\|_{\boldsymbol L^2(\tilde{\Omega})},\quad
\|\boldsymbol{u}\|_{\boldsymbol H^1(\Omega)}^2 =
\|\boldsymbol{u}\|_{\boldsymbol L^2(\Omega)}^2 +\|\nabla
\boldsymbol{u}\|_{\boldsymbol L^2(\Omega)}^2.
\]
\end{proof}

\begin{theorem}\label{wp}
The variational problem \eqref{vp} admits a unique weak solution
$\boldsymbol{u}\in\boldsymbol H^1_{\partial D}(\Omega)$.
\end{theorem}

\begin{proof}
Using the Cauchy--Schwarz inequality, Lemma \ref{tbo}, and Lemma \ref{tr}, we
have
\begin{align*}
|b(\boldsymbol{u}, \boldsymbol{v})| \leq &\mu
\|\nabla\boldsymbol{u}\|_{\boldsymbol
L^2(\Omega)}\|\nabla\boldsymbol{v}\|_{\boldsymbol L^2(\Omega)}
+(\lambda+\mu)\|\nabla \cdot \boldsymbol{u}\|_{0,\Omega}\|\nabla \cdot
\boldsymbol{v}\|_{\boldsymbol L^2(\Omega)} + \omega^2
\|\boldsymbol{u}\|_{\boldsymbol L^2(\Omega)}\|\boldsymbol{v}\|_{\boldsymbol
L^2(\Omega)} \\
&+\|\mathscr{T}\boldsymbol{u}\|_{\boldsymbol
H^{-1/2}(\Gamma_R)}\|\boldsymbol{v}\|_{\boldsymbol H^{1/2}(\Gamma_R) }
\notag \\
\lesssim &\|\boldsymbol{u}\|_{\boldsymbol H^1(\Omega)}
\|\boldsymbol{v}\|_{\boldsymbol H^1(\Omega)},
\end{align*}
which shows that the sesquilinear form $b(\cdot, \cdot)$ is bounded.

It follows from Lemma \ref{mpd} that there exists an $N_0 \in \mathbb{N}$ such
that $\hat{M}_n$ is positive definite for $n>N_0$. The sesquilinear form $b$
can be written as
\begin{align*}
  b(\boldsymbol{u}, \boldsymbol{v}) =& \mu  \int_\Omega (\nabla \boldsymbol{u}
:\nabla \bar{\boldsymbol{v}})\,{\rm d}\boldsymbol{x} +
(\lambda+\mu)\int_\Omega(\nabla \cdot \boldsymbol{u}) (\nabla \cdot
\bar{\boldsymbol{v}}) \, {\rm d}\boldsymbol{x} - \omega^2 \int_\Omega
\boldsymbol{u}\cdot\bar{\boldsymbol{v}}\,{\rm d}\boldsymbol{x}\\
&\qquad -\sum_{|n| > N_0}\sum_{m=-n}^{n} \left\langle M_n
\boldsymbol{u}_n^m, \boldsymbol{v}_n^m \right\rangle
 - \sum_{|n| \leq N_0}\sum_{m=-n}^{n} \left\langle M_n
\boldsymbol{u}_n^m, \boldsymbol{v}_n^m \right\rangle.
\end{align*}
Taking the real part of $b$, and using Lemma \ref{mpd}, Lemma
\ref{po}, Lemma \ref{tr0}, we obtain
\begin{align*}
  {\rm Re}\, b(\boldsymbol{u}, \boldsymbol{u}) &= \mu \|\nabla
\boldsymbol{u} \|_{\boldsymbol L^2(\Omega)}^2 + (\lambda+\mu) \|\nabla \cdot
\boldsymbol{u}\|_{\boldsymbol L^2(\Omega)}^2 + \sum_{|n| > N_0}\sum_{m=-n}^{n}
\langle \hat{M}_n \boldsymbol{u}_n^m, \boldsymbol{u}_n^m\rangle \\
&\qquad-\omega^2 \|\boldsymbol{u}\|_{\boldsymbol L^2(\Omega)} + \sum_{|n|\leq
N_0} \sum_{m=-n}^{n}\langle \hat{M}_n
\boldsymbol{u}_n^m, \boldsymbol{u}_n^m\rangle\\
&\geq C_1\|\boldsymbol{u}\|_{\boldsymbol H^1(\Omega)}-\omega^2
\|\boldsymbol{u}\|_{\boldsymbol L^2(\Omega)}
-C_2\|\boldsymbol{u}\|_{\boldsymbol L^2(\Gamma_R)}\\
&\geq C_1\|\boldsymbol{u}\|_{\boldsymbol H^1(\Omega)}-\omega^2
\|\boldsymbol{u}\|_{\boldsymbol L^2(\Omega)} -C_2\varepsilon
\|\boldsymbol{u}\|_{\boldsymbol
H^1(\Omega)}-C(\varepsilon)\|\boldsymbol{u}\|_{\boldsymbol L^2(\Omega)}\\
&=(C_1-C_2\varepsilon)\|\boldsymbol{u}\|_{\boldsymbol H^1(\Omega)}
-C_3\|\boldsymbol{u}\|_{\boldsymbol L^2(\Omega)}.
\end{align*}
Letting $\varepsilon>0$ to be sufficiently small, we have $C_1-C_2\varepsilon>0$
and thus G{\aa}rding's inequality. Since the injection of
$\boldsymbol H^1_{\partial D}(\Omega)$ into $\boldsymbol L^2(\Omega)$ is
compact, the proof is completed by using the Fredholm alternative (cf.
\cite[Theorem 5.4.5]{n-00}) and the uniqueness
result in Theorem \ref{uni}.
\end{proof}

\section{Inverse scattering}

In this section, we study a domain derivative of the scattering problem and
present a continuation method to reconstruct the surface.

\subsection{Domain derivative}

We assume that the obstacle has a $C^2$ boundary, i.e., $\partial D\in C^2$.
Given a sufficiently small number $h>0$, define a perturbed domain $\Omega_h$
which is surrounded by $\partial D_h$ and $\Gamma_R$, where
\[
\partial D_h=\{\boldsymbol{x}+h\boldsymbol{p(x)}:\boldsymbol{x}\in\partial D\}.
\]
Here the function $\boldsymbol{p}\in\boldsymbol C^2(\partial D)$.

Consider the variational formulation for the direct problem in the perturbed
domain $\Omega_h$: To find $\boldsymbol{u}_h\in\boldsymbol H^{1}_{\partial
D_h}(\Omega_h)$ such that
\begin{equation}\label{vph}
 b^h(\boldsymbol{u}_h,
\boldsymbol{v}_h)=\langle\boldsymbol{g},\boldsymbol{v}_h\rangle_{\Gamma_R},
\quad\forall\, \boldsymbol{v}_h\in\boldsymbol H^1_{\partial D_h}(\Omega_h),
\end{equation}
where the sesquilinear form $b^h: \boldsymbol H^1_{\partial
D_h}(\Omega_h)\times\boldsymbol H^1_{\partial D_h}(\Omega_h)\to\mathbb{C}$ is
defined by
\begin{align}\label{sfh}
 b^h(\boldsymbol{u}_h, \boldsymbol{v}_h)=\mu \int_{\Omega_h}
\nabla\boldsymbol{u}_h: \nabla\bar{\boldsymbol v}_h\,{\rm
d}\boldsymbol{y}+(\lambda+\mu)\int_{\Omega_h}
(\nabla\cdot\boldsymbol{u}_h)(\nabla\cdot\bar{\boldsymbol v}_h)\,{\rm
d}\boldsymbol{y}\notag\\
-\omega^2\int_{\Omega_h} \boldsymbol{u}_h\cdot\bar{\boldsymbol v}_h\,{\rm
d}\boldsymbol{y}-\langle\mathscr{T}\boldsymbol{u}_h,
\boldsymbol{v}_h\rangle{\Gamma_R}.
\end{align}
Similarly, we may follow the proof of Theorem \ref{wp} to show that the
variational problem \eqref{vph} has a unique weak solution $\boldsymbol{u}_h\in
\boldsymbol H^1_{\partial D_h}(\Omega_h)$ for any $h>0$.

Since the variational problem \eqref{wp} is well-posed, we introduce a nonlinear
scattering operator:
\[
\mathscr{S}:\partial D_h\rightarrow \boldsymbol{u}_h|_{\Gamma_R},
\]
which maps the obstacle's surface to the displacement of the wave field on
$\Gamma_R$. Let $\boldsymbol u_h$ and $\boldsymbol u$ be the solution of the
direct problem in the domain $\Omega_h$ and $\Omega$, respectively. Define the
domain derivative of the scattering operator $\mathscr{S}$ on $\partial D$ along
the direction $\boldsymbol{p}$ as
\[
\mathscr{S}'(\partial D;\boldsymbol{p}):=\lim_{h\rightarrow
0}\frac{\mathscr{S}(\partial D_h)-\mathscr{S}(\partial D)}{h}
=\lim_{h\rightarrow=0}\frac{\boldsymbol{u}_h|_{\Gamma_R}-\boldsymbol{u}|_{\Gamma
_R}}{h} .
\]

For a given $\boldsymbol{p}\in \boldsymbol C^2(\partial D)$, we extend its
domain to $\bar{\Omega}$ by requiring that $\boldsymbol{p}\in
\boldsymbol C^2(\Omega)\cap \boldsymbol C(\bar{\Omega}), \boldsymbol{p}=0$
on $\Gamma_R$,
and
$\boldsymbol{y}=\boldsymbol\xi^h(\boldsymbol{x})=\boldsymbol{x}+h\boldsymbol{p}
(\boldsymbol{x})$ maps $\Omega$ to $\Omega_h$. It is clear to note that
$\boldsymbol\xi^h$ is a diffeomorphism from $\Omega$ to $\Omega_h$ for
sufficiently small $h$. Denote by $\boldsymbol\eta^h(\boldsymbol{y}):
\Omega_h\to\Omega$ the inverse map of $\boldsymbol\xi^h$.

Define $\breve{\boldsymbol{u}}(\boldsymbol{x})=(\breve{u}_1,
\breve{u}_2, \breve{u}_{3}):=(\boldsymbol{u}
_h\circ\boldsymbol\xi^h)(\boldsymbol{x})$. Using the change of variable
$\boldsymbol{y}=\xi^h(\boldsymbol{x})$, we have from straightforward
calculations that 
\begin{align*}
 \int_{\Omega_h}(\nabla\boldsymbol{u}_h: \nabla\overline{\boldsymbol v}_h)\,{\rm
d}\boldsymbol{y}&=\sum_{j=1}^3\int_\Omega \nabla\breve{u}_j
J_{\boldsymbol\eta^h} J_{\boldsymbol\eta^h}^\top \nabla\bar{\breve{v}}_j\,{\rm
det}(J_{\boldsymbol\xi^h})\,{\rm d}\boldsymbol{x},\\
\int_{\Omega_h}(\nabla\cdot\boldsymbol{u}_h)(\nabla\cdot\bar{\boldsymbol
v}_h)\,{ \rm d}\boldsymbol{y}&=\int_\Omega (\nabla\breve{\boldsymbol u}:
J_{\boldsymbol\eta^h}^\top) (\nabla\bar{\breve{\boldsymbol v}}:
J_{\boldsymbol\eta^h}^\top)\,{\rm det}(J_{\boldsymbol\xi^h})\,{\rm
d}\boldsymbol{x},\\
\int_{\Omega_h} \boldsymbol{u}_h\cdot\bar{\boldsymbol v}_h\,{ \rm
d}\boldsymbol{y}&=\int_\Omega
\breve{\boldsymbol{u}}\cdot\bar{\breve{\boldsymbol{
v}}}\,{\rm det}(J_{\boldsymbol\xi^h})\,{\rm d}\boldsymbol{x},
\end{align*}
where $\breve{\boldsymbol{v}}(\boldsymbol{x})=(\breve{v}_1,
\breve{v}_2,
\breve{v}_3):=(\boldsymbol{v}_h\circ\boldsymbol\xi^h)(\boldsymbol{x})$,
$J_{\boldsymbol\eta^h}$ and $J_{\boldsymbol\xi^h}$ are the Jacobian
matrices of the transforms $\boldsymbol\eta^h$ and
$\boldsymbol\xi^h$, respectively.

For a test function $\boldsymbol{v}_h$ in the domain $\Omega_h$, it
follows from the transform that $\breve{\boldsymbol v}$ is a test function in
the domain $\Omega$. Therefore, the sesquilinear form $b^h$ in
\eqref{sfh} becomes
\begin{align*}
 b^h(\breve{\boldsymbol u}, \boldsymbol{v})=\sum_{j=1}^3 \mu \int_\Omega
\nabla\breve{u}_j J_{\boldsymbol\eta^h} J_{\boldsymbol\eta^h}^\top
\nabla\bar{v}_j\,{\rm det}(J_{\boldsymbol\xi^h})\,{\rm
d}\boldsymbol{x}+(\lambda+\mu)
\int_\Omega
(\nabla\breve{\boldsymbol u}: J_{\boldsymbol\eta^h}^\top) 
(\nabla\bar{\boldsymbol v}: J_{\boldsymbol\eta^h}^\top)\\
\times{\rm det}(J_{\boldsymbol\xi^h})\,{\rm d}\boldsymbol{x}
-\omega^2 \int_\Omega \breve{\boldsymbol{u}}\cdot\bar{\boldsymbol{v}}\,{\rm
det}(J_{\boldsymbol\xi^h})\,{\rm
d}\boldsymbol{x}-\langle\mathscr{T}\breve{\boldsymbol{u}},
\boldsymbol{v}\rangle_{\Gamma_R},
\end{align*}
which gives an equivalent variational formulation of \eqref{vph}:
\[
 b^h(\breve{\boldsymbol u}, \boldsymbol{v})=\langle \boldsymbol{g},
\boldsymbol{v}\rangle_{\Gamma_R},\quad\forall\, \boldsymbol{v}\in
\boldsymbol H^1_{\partial D}(\Omega).
\]

A simple calculation yields
\begin{align*}
 b(\breve{\boldsymbol u}-\boldsymbol{u}, \boldsymbol{v})=b(\breve{\boldsymbol
u}, \boldsymbol{v})-\langle\boldsymbol{g},
\boldsymbol{v}\rangle_{\Gamma_R}=b(\breve{\boldsymbol u},
\boldsymbol{v})-b^h(\breve{\boldsymbol
u}, \boldsymbol{v})=b_1 + b_2 + b_3,
\end{align*}
where
\begin{align}
\label{b1}b_1&=\sum_{j=1}^3 \mu \int_\Omega \nabla\breve{u}_j \left(I-
J_{\boldsymbol\eta^h} J_{\boldsymbol\eta^h}^\top\,{\rm
det}(J_{\boldsymbol\xi^h})\right) \nabla\bar{v}_j\,{\rm
d}\boldsymbol{x},\\
\label{b2}b_2&=(\lambda+\mu)\int_\Omega (\nabla\cdot\breve{\boldsymbol
u})(\nabla\cdot\bar{\boldsymbol v})-(\nabla\breve{\boldsymbol u}:
J_{\boldsymbol\eta^h}^\top) (\nabla\bar{\boldsymbol v}:
J_{\boldsymbol\eta^h}^\top)\,{\rm det}(J_{\boldsymbol\xi^h})\,{\rm
d}\boldsymbol{x},\\
\label{b3}b_3&=\omega^2\int_\Omega
\breve{\boldsymbol{u}}\cdot\bar{\boldsymbol{v}}\,\left({ \rm
det}(J_{\boldsymbol\xi^h})-1\right)\,{\rm d}\boldsymbol{x}.
\end{align}
Here $I$ is the identity matrix. Following the definitions of the Jacobian
matrices, we may easily verify that 
\begin{align*}
{\rm
det}(J_{\boldsymbol\xi^h})&=1+h\nabla\cdot\boldsymbol{p}+O(h^2),\\
J_{\boldsymbol\eta^h}&=J^{-1}_{\boldsymbol\xi^h}\circ\boldsymbol\eta^h=I-h
J_{\boldsymbol p}+O(h^2),\\
J_{\boldsymbol\eta^h} J^\top_{\boldsymbol\eta^h} {\rm
det}(J_{\boldsymbol\xi^h})&=I-h(J_{\boldsymbol p}+J^\top_{\boldsymbol p}
)+h(\nabla\cdot\boldsymbol{p})I+O(h^2),
\end{align*}
where the matrix $J_{\boldsymbol p}=\nabla \boldsymbol{p}$.

Substituting the above estimates into \eqref{b1}--\eqref{b3}, we obtain
\begin{align*}
 b_1&=\sum_{j=1}^3 \mu \int_\Omega \nabla\breve{u}_j \left( h(J_{\boldsymbol
p}+J^\top_{\boldsymbol p}
)-h(\nabla\cdot\boldsymbol{p})I+O(h^2)\right)
\nabla\bar{v}_j\,{\rm d}\boldsymbol{x},\\
b_2&=(\lambda+\mu)\int_\Omega h (\nabla\cdot\breve{\boldsymbol
u})(\nabla\bar{\boldsymbol v}: J^\top_{\boldsymbol
p})+h(\nabla\cdot\bar{\boldsymbol v})(\nabla\breve{\boldsymbol u}:
J^\top_{\boldsymbol p})\\
&\hspace{3cm}-h(\nabla\cdot\boldsymbol{p})(\nabla\cdot\breve{
\boldsymbol u})(\nabla\cdot\bar{\boldsymbol v})+O(h^2)\,{\rm
d}\boldsymbol{x},\\
b_3&=\omega^2\int_\Omega \breve{\boldsymbol{u}}\cdot\bar{\boldsymbol{v}}\,
\left(h\nabla\cdot\boldsymbol{p}+O(h^2)\right)\,{\rm d}\boldsymbol{x}.
\end{align*}
Hence we have
\begin{equation}\label{dvp}
b\left(\frac{\breve{\boldsymbol u}-\boldsymbol{u}}{h}, \boldsymbol{v}
\right)=g_1(\boldsymbol{p})(\breve{\boldsymbol u},
\boldsymbol{v})+g_2(\boldsymbol{p})(\breve{\boldsymbol u},
\boldsymbol{v})+g_3(\boldsymbol{p})(\breve{\boldsymbol u},
\boldsymbol{v})+O(h),
\end{equation}
where
\begin{align*}
 g_1&=\sum_{j=1}^3 \mu \int_\Omega \nabla\breve{u}_j \left( (J_{\boldsymbol
p}+J^\top_{\boldsymbol p} )-(\nabla\cdot\boldsymbol{p})I\right)
\nabla\bar{v}_j\,{\rm d}\boldsymbol{x},\\
g_2&=(\lambda+\mu)\int_\Omega(\nabla\cdot\breve{\boldsymbol
u})(\nabla\bar{\boldsymbol v}: J^\top_{\boldsymbol
p})+(\nabla\cdot\bar{\boldsymbol v})(\nabla\breve{\boldsymbol u}:
J^\top_{\boldsymbol p})-(\nabla\cdot\boldsymbol{p})(\nabla\cdot\breve{
\boldsymbol u})(\nabla\cdot\bar{\boldsymbol v})\,{\rm d}\boldsymbol{x},\\
g_3&=\omega^2\int_\Omega (\nabla\cdot\boldsymbol{p})
\breve{\boldsymbol{u}}\cdot\bar{\boldsymbol{v}}\,{\rm d}\boldsymbol{x}.
\end{align*}

\begin{theorem}\label{dd}
Given $\boldsymbol{p}\in \boldsymbol C^2(\partial D)$, the domain derivative of
the scattering operator $\mathscr{S}$ is $\mathscr{S}'(\partial D;
\boldsymbol{p})=\boldsymbol{u}'|_{\Gamma_R}$, where $\boldsymbol{u}'$ is the
unique weak solution of the boundary value problem:
\begin{align}\label{dbvp}
\begin{cases}
\mu\Delta\boldsymbol{u}'+(\lambda+\mu)\nabla\nabla\cdot\boldsymbol{u}'
+\omega^2\boldsymbol{u}'=0 &\quad\text{in} ~ \Omega,\\
\boldsymbol{u}'=-({\boldsymbol p}\cdot\boldsymbol\nu)\partial_{\boldsymbol\nu}
\boldsymbol{u} &\quad\text{on} ~ \partial D,\\
\mathscr{B}\boldsymbol{u}'=\mathscr{T}\boldsymbol{u}'
&\quad\text{on} ~\Gamma_R,
\end{cases}
\end{align}
and $\boldsymbol{u}$ is the solution of the variational problem \eqref{vp}
corresponding to the domain $\Omega$.
\end{theorem}

\begin{proof}
Given $\boldsymbol{p}\in\boldsymbol C^2(\partial D)$, we extend its
definition to the domain $\bar{\Omega}$ as before. It follows from the
well-posedness of the variational
problem \eqref{vp} that $\breve{\boldsymbol u}\to\boldsymbol{u}$ in
$\boldsymbol H^1_{\partial D}(\Omega)$ as $h\to 0$. Taking the limit $h\to 0$
in \eqref{dvp} gives
\begin{equation}\label{lvp}
 b\left(\lim_{h\to 0}\frac{\breve{\boldsymbol u}-\boldsymbol{u}}{h},
\boldsymbol{v}\right)=g_1(\boldsymbol{p})(\boldsymbol{u},
\boldsymbol{v})+g_2(\boldsymbol{p})(\boldsymbol{u},
\boldsymbol{v})+g_3(\boldsymbol{p})(\boldsymbol{u},
\boldsymbol{v}),
\end{equation}
which shows that $(\breve{\boldsymbol u}-\boldsymbol{u})/h$ is convergent in
$\boldsymbol H^1_{\partial D}(\Omega)$ as $h\to 0$. Denote the limit by 
$\dot{\boldsymbol u}$ and rewrite \eqref{lvp} as
\begin{equation}\label{lbvp}
 b(\dot{\boldsymbol u}, \boldsymbol{v})=g_1(\boldsymbol{p})(\boldsymbol{u},
\boldsymbol{v})+g_2(\boldsymbol{p})(\boldsymbol{u},
\boldsymbol{v})+g_3(\boldsymbol{p})(\boldsymbol{u},
\boldsymbol{v}).
\end{equation}

First we compute $g_1(\boldsymbol{p})(\boldsymbol{u}, \boldsymbol{v})$. Noting
$\boldsymbol{p}=0$ on $\partial B$ and using the identity
\begin{align*}
 \nabla u \left( (J_{\boldsymbol p}+J^\top_{\boldsymbol
p})-(\nabla\cdot\boldsymbol{p})I \right)\nabla\bar{v}=&\nabla\cdot\left[
(\boldsymbol{p}\cdot\nabla u)\nabla\bar{v}+(\boldsymbol{p}
\cdot\nabla\bar{v})\nabla u -
(\nabla u\cdot\nabla\bar{v})\boldsymbol{p}\right]\\
&-(\boldsymbol{p}\cdot\nabla u)\Delta\bar{v}-(\boldsymbol{p}
\cdot\nabla\bar{v})\Delta u,
\end{align*}
we obtain from the divergence theorem that
\begin{align*}
 g_1(\boldsymbol{p})(\boldsymbol{u}, \boldsymbol{v})&=-\sum_{j=1}^3
    \mu\int_\Omega (\boldsymbol{p}\cdot\nabla u_j)\Delta\bar{v}_j
    +(\boldsymbol{p}\cdot\nabla\bar{v}_j)\Delta u_j\,{\rm
d}\boldsymbol{x}\\
 &\qquad -\sum_{j=1}^3 \mu\int_{\partial D}(\boldsymbol{p}\cdot\nabla u_j)   
(\boldsymbol\nu\cdot\nabla\bar{v}_j)+(\boldsymbol{p}\cdot\nabla\bar{v}_j)
    (\boldsymbol\nu\cdot\nabla u_j) - (\boldsymbol{p}\cdot\boldsymbol\nu)
    (\nabla u_j\cdot\nabla\bar{v}_j)\, {\rm d}\gamma\\
 &=-\mu\int_\Omega (\boldsymbol{p}\cdot\nabla\boldsymbol{u})\cdot\Delta
    \bar{\boldsymbol v}+(\boldsymbol{p}\cdot\nabla\bar{\boldsymbol v})
    \cdot\Delta \boldsymbol{u}\,{\rm d}\boldsymbol{x}\\
 &\qquad -\mu\int_{\partial D}(\boldsymbol{p}\cdot\nabla\boldsymbol{u})
    \cdot(\boldsymbol\nu\cdot\nabla\bar{\boldsymbol v})+(\boldsymbol{p}
    \cdot\nabla\bar{\boldsymbol
v})\cdot(\boldsymbol\nu\cdot\nabla\boldsymbol{u}) -
    (\boldsymbol{p}\cdot\boldsymbol\nu)(\nabla \boldsymbol{u}:
    \nabla\bar{\boldsymbol v})\, {\rm d}\gamma.
\end{align*}
Noting 
\[
\mu\Delta\boldsymbol{u}+(\lambda+\mu)\nabla\nabla\cdot\boldsymbol{u}
+\omega^2\boldsymbol{u}=0\quad \text{in} ~\Omega,
\]
we have from the integration by parts that 
\begin{align*}
& \mu\int_\Omega (\boldsymbol{p}\cdot\nabla\bar{\boldsymbol v}
)\cdot\Delta \boldsymbol{u}\,{\rm d}\boldsymbol{x}=-(\lambda+\mu)\int_\Omega
(\boldsymbol{p}\cdot\nabla\bar{\boldsymbol
v})\cdot(\nabla\nabla\cdot\boldsymbol{u})\,{\rm
d}\boldsymbol{x}-\omega^2\int_\Omega
(\boldsymbol{p}\cdot\nabla\bar{\boldsymbol
v})\cdot\boldsymbol{u}\,{\rm d}\boldsymbol{x}\\
&=(\lambda+\mu)\int_\Omega (\nabla\cdot\boldsymbol{u})
\nabla\cdot(\boldsymbol{p}\cdot\nabla\bar{\boldsymbol
v})\,{\rm d}\boldsymbol{x}+(\lambda+\mu)\int_{\partial
D}(\nabla\cdot\boldsymbol{u})(\boldsymbol\nu\cdot
(\boldsymbol{p}\cdot\nabla\bar{\boldsymbol
v}) )\,{\rm d}\gamma \\
&\hspace{4cm}-\omega^2\int_\Omega
(\boldsymbol{p}\cdot\nabla\bar{\boldsymbol
v})\cdot\boldsymbol{u}\,{\rm d}\boldsymbol{x}.
\end{align*}
Using the integration by parts again yields
\[
 \mu\int_\Omega (\boldsymbol{p}\cdot\nabla
\boldsymbol{u})\cdot\Delta\bar{\boldsymbol
v}\,{\rm d}\boldsymbol{x}=-\mu\int_\Omega \nabla(\boldsymbol{p}\cdot\nabla
\boldsymbol{u}): \nabla\bar{\boldsymbol v}\,{\rm
d}\boldsymbol{x}+\mu\int_{\partial D} (\boldsymbol{p}\cdot\nabla
\boldsymbol{u})\cdot (\boldsymbol\nu\cdot\nabla\bar{\boldsymbol v})\,{\rm
d}\gamma.
\]
Let $\boldsymbol\tau_1(\boldsymbol{x}), \boldsymbol\tau_2(\boldsymbol{x})$
be any two linearly independent unit tangent vectors on $\partial D$. Since
$\boldsymbol{u}=\boldsymbol{v}=0$ on
$\partial D$, we have
\[
\partial_{\boldsymbol\tau_1}u_j=\partial_{\boldsymbol\tau_2}u_j=
\partial_{\boldsymbol\tau_1}v_j=\partial_{\boldsymbol\tau_2}v_j=0.
\]
Using the identities
\begin{align*}
\nabla u_j&=\boldsymbol\tau_1\partial_{\boldsymbol\tau_1}
u_j+\boldsymbol\tau_2\partial_{\boldsymbol\tau_2} u_j
+\boldsymbol\nu\partial_{\boldsymbol\nu}
u_j=\boldsymbol\nu\partial_{\boldsymbol\nu} u_j,\\
\nabla v_j&=\boldsymbol\tau_1\partial_{\boldsymbol\tau_1}
v_j+\boldsymbol\tau_2\partial_{\boldsymbol\tau_2} v_j
+\boldsymbol\nu\partial_{\boldsymbol\nu}
v_j=\boldsymbol\nu\partial_{\boldsymbol\nu} v_j,
\end{align*}
we have
\[
(\boldsymbol{p}\cdot\nabla\bar{v}_j)(\boldsymbol\nu\cdot\nabla u_j) 
 =(\boldsymbol{p}\cdot\boldsymbol\nu\partial_{\boldsymbol\nu}\bar{v}_j
)(\boldsymbol\nu\cdot\boldsymbol\nu\partial_{\boldsymbol\nu} u_j)   
=(\boldsymbol{p}\cdot\boldsymbol\nu)(\partial_{\boldsymbol\nu}\bar{v}_j
\partial_{\boldsymbol\nu}u_j),
\]
which gives 
\[
 \int_{\partial D}(\boldsymbol{p}
\cdot\nabla\bar{\boldsymbol v
})\cdot(\boldsymbol\nu\cdot\nabla\boldsymbol{u})
-(\boldsymbol{p}\cdot\boldsymbol\nu)(\nabla \boldsymbol{u}:
\nabla\bar{\boldsymbol v})\, {\rm d}\gamma=0.
\]
Noting $\boldsymbol{v}=0$ on $\partial D$ and
\[
 (\nabla\cdot\boldsymbol{p})(\boldsymbol{u}\cdot\bar{\boldsymbol
v})+(\boldsymbol{p}\cdot\nabla\bar{\boldsymbol
v})\cdot\boldsymbol{u}=\nabla\cdot((\boldsymbol{u}\cdot\bar{\boldsymbol
v})\boldsymbol{p})-(\boldsymbol{p}\cdot\nabla\boldsymbol{u})\cdot\bar{
\boldsymbol v},
\]
we obtain by the divergence theorem that
\[
 \int_\Omega 
(\nabla\cdot\boldsymbol{p})(\boldsymbol{u}\cdot\bar{\boldsymbol
v})+(\boldsymbol{p}\cdot\nabla\bar{\boldsymbol
v})\cdot\boldsymbol{u}\,{\rm d}\boldsymbol{x}=-\int_\Omega
(\boldsymbol{p}\cdot\nabla\boldsymbol{u})\cdot\bar{
\boldsymbol v}\,{\rm d}\boldsymbol{x}.
\]
Combining the above identities, we conclude that
\begin{align}\label{g1g3}
  g_1(\boldsymbol{p})(\boldsymbol{u}, \boldsymbol{v})+
g_3(\boldsymbol{p})(\boldsymbol{u}, \boldsymbol{v})=\mu\int_\Omega
\nabla(\boldsymbol{p}\cdot\nabla \boldsymbol{u}): \nabla\bar{\boldsymbol
v}\,{\rm d}\boldsymbol{x}-(\lambda+\mu)\int_\Omega (\nabla\cdot\boldsymbol{u})
\nabla\cdot(\boldsymbol{p}\cdot\nabla\bar{\boldsymbol
v})\,{\rm d}\boldsymbol{x}\notag\\
-\omega^2\int_\Omega (\boldsymbol{p}\cdot\nabla\boldsymbol{u})\cdot\bar{
\boldsymbol v}\,{\rm d}\boldsymbol{x}+(\lambda+\mu)\int_{\partial
D}(\nabla\cdot\boldsymbol{u})(\boldsymbol\nu\cdot(\boldsymbol{
p}\cdot\nabla\bar{\boldsymbol v}))\,{ \rm d}\gamma.
\end{align}

Next we compute $g_2(\boldsymbol{p})(\boldsymbol{u}, \boldsymbol{v})$. It is
easy to verify that
\begin{align*}
\int_\Omega &(\nabla\cdot\boldsymbol{u}) (\nabla\bar{\boldsymbol v}:
J_{\boldsymbol p}^\top)+(\nabla\cdot\bar{\boldsymbol v})
(\nabla\boldsymbol{u}: J_{\boldsymbol p}^\top)\,{\rm
d}\boldsymbol{x}=\int_\Omega
(\nabla\cdot\boldsymbol{u})\nabla\cdot(\boldsymbol{p}\cdot\nabla\bar{
\boldsymbol v})\,{\rm d}\boldsymbol{x}\\
&\quad-\int_\Omega (\nabla\cdot\boldsymbol{u})(\boldsymbol{p}
\cdot(\nabla\cdot (\nabla\bar{\boldsymbol{v}})^\top) )\,{\rm
d}\boldsymbol{x}+\int_\Omega (\nabla\cdot\bar{\boldsymbol
v})\nabla\cdot(\boldsymbol{p}\cdot\nabla\boldsymbol{u})\,{\rm
d}\boldsymbol{x}\\
&\hspace{4cm}-\int_\Omega (\nabla\cdot\bar{\boldsymbol
v})(\boldsymbol{p} \cdot(\nabla\cdot (\nabla\boldsymbol{u})^\top))\,{\rm
d}\boldsymbol{x}.
\end{align*}
Using the integration by parts, we obtain
\begin{align*}
&\int_\Omega
(\nabla\cdot\boldsymbol{p})(\nabla\cdot\boldsymbol{u}
)(\nabla\cdot\bar{\boldsymbol v})\,{\rm d}\boldsymbol{x}=-\int_\Omega
\boldsymbol{p}\cdot\nabla ((\nabla\cdot\boldsymbol{u}
)(\nabla\cdot\bar{\boldsymbol v}))\,{\rm d}\boldsymbol{x}\\
&\hspace{6cm}-\int_{\partial
D}(\nabla\cdot\boldsymbol{u})(\nabla\cdot\bar{\boldsymbol
v})(\boldsymbol\nu\cdot\boldsymbol{p})\,{\rm d}\gamma\\
&=-\int_\Omega (\nabla\cdot\bar{\boldsymbol v})(\boldsymbol{p}
\cdot(\nabla\cdot (\nabla\boldsymbol{u})^\top))\,{\rm
d}\boldsymbol{x}-\int_\Omega
(\nabla\cdot\boldsymbol{u})(\boldsymbol{p}
\cdot(\nabla\cdot (\nabla\boldsymbol{v})^\top))\,{\rm
d}\boldsymbol{x}\\
&\hspace{6cm}-\int_{\partial D}(\nabla\cdot\boldsymbol{u})(\nabla\cdot\bar{
\boldsymbol v})(\boldsymbol\nu\cdot\boldsymbol{p})\,{\rm d}\gamma.
\end{align*}
Let $
\boldsymbol\tau_1=(-\nu_3,0,\nu_1)^\top,
\boldsymbol\tau_2=(0,-\nu_3,\nu_2)^\top,
\boldsymbol\tau_3=(-\nu_2,\nu_1,0)^\top.$
It follows from $\boldsymbol\tau_j\cdot\boldsymbol\nu=0$ that
$\boldsymbol\tau_j$ are tangent vectors on
$\partial D$. Since $\boldsymbol{v}=0$ on $\partial D$, we have 
$\partial_{\boldsymbol\tau_j}\boldsymbol{v}=0$, which yields that 
\begin{align*}
& \nu_1\partial_{x_3}v_1=\nu_3\partial_{x_1} v_1,\quad
 \nu_1\partial_{x_3}v_2=\nu_3\partial_{x_1} v_2, \quad
 \nu_1\partial_{x_2}v_1=\nu_2\partial_{x_1} v_1,\\
& \nu_1\partial_{x_3}v_3=\nu_3\partial_{x_1} v_3,\quad
 \nu_1\partial_{x_2}v_2=\nu_2\partial_{x_1} v_2,\quad
 \nu_1\partial_{x_2}v_3=\nu_2\partial_{x_1} v_3,\\
& \nu_2\partial_{x_3}v_1=\nu_3\partial_{x_2} v_1,\quad
 \nu_2\partial_{x_3}v_2=\nu_3\partial_{x_2} v_2,\quad
 \nu_2\partial_{x_3}v_3=\nu_3\partial_{x_2} v_3.
\end{align*}
Hence we get
\[
 \int_{\partial D}(\nabla\cdot\boldsymbol{u})(\nabla\cdot\bar{\boldsymbol
v})(\boldsymbol\nu\cdot\boldsymbol{p})\,{\rm d}\gamma=\int_{\partial
D}(\nabla\cdot\boldsymbol{u})(\boldsymbol\nu\cdot(\boldsymbol{p}
\cdot\nabla\bar{\boldsymbol v}))\,{ \rm d}\gamma.
\]
Combining the above identities gives
\begin{align}\label{g2}
 g_2(\boldsymbol{p})(\boldsymbol{u}, \boldsymbol{v})=(\lambda+\mu)&\int_\Omega
(\nabla\cdot\boldsymbol{u})\nabla\cdot(\boldsymbol{p}\cdot\nabla\bar{
\boldsymbol v})\,{\rm d}\boldsymbol{x} + (\lambda+\mu)\int_\Omega
\nabla\cdot(\boldsymbol{p}\cdot\nabla\boldsymbol{u})(\nabla\cdot\bar{
\boldsymbol v})\,{\rm d}\boldsymbol{x}\notag\\
&-(\lambda+\mu)\int_{\partial
D}(\nabla\cdot\boldsymbol{u})(\nu\cdot(\boldsymbol{
p}\cdot\nabla\bar{\boldsymbol v}))\,{ \rm d}\gamma.
\end{align}

Noting \eqref{lbvp},  adding \eqref{g1g3} and \eqref{g2}, we obtain
\[
 b(\dot{\boldsymbol u}, \boldsymbol{v})=\mu\int_\Omega
\nabla(\boldsymbol{p}\cdot\nabla \boldsymbol{u}): \nabla\bar{\boldsymbol
v}\,{\rm d}\boldsymbol{x} + (\lambda+\mu)\int_\Omega
\nabla\cdot(\boldsymbol{p}\cdot\nabla\boldsymbol{u})(\nabla\cdot\bar{
\boldsymbol v})\,{\rm d}\boldsymbol{x}-\omega^2\int_\Omega
(\boldsymbol{p}\cdot\nabla\boldsymbol{u})\cdot\bar{
\boldsymbol v}\,{\rm d}\boldsymbol{x}.
\]
Define $\boldsymbol{u}'=\dot{\boldsymbol
u}-\boldsymbol{p}\cdot\nabla\boldsymbol{u}$. It is clear to note that
$\boldsymbol{p}\cdot\nabla\boldsymbol{u}=0$ on $\Gamma_R$ since
$\boldsymbol{p}=0$ on $\Gamma_R$. Hence, we have
\begin{equation}\label{vpd}
  b({\boldsymbol u}', \boldsymbol{v})=0,\quad\forall\,\boldsymbol{v}\in
\boldsymbol H^1_{\partial D}(\Omega),
\end{equation}
which shows that $\boldsymbol{u}'$ is the weak solution of the boundary value
problem \eqref{dbvp}. To verify the boundary condition of $\boldsymbol{u}'$ on
$\partial D$, we recall the definition of $\boldsymbol{u}'$ and have from
$\breve{\boldsymbol u}=\boldsymbol{u}=0$ on $\partial D$ that 
\[
 \boldsymbol{u}'=\lim_{h\to 0}\frac{\breve{\boldsymbol
u}-\boldsymbol{u}}{h}-\boldsymbol{p}\cdot\nabla\boldsymbol{u}=-\boldsymbol{p}
\cdot\nabla\boldsymbol{u}\quad\text{on} ~ \partial D.
\]
Noting $\boldsymbol{u}=0$ on $\partial D$, we have 
\begin{equation}\label{dbc}
 {\boldsymbol p}\cdot\nabla\boldsymbol{u}=({\boldsymbol
p}\cdot\boldsymbol\nu)\partial_{\boldsymbol\nu}\boldsymbol{u},
\end{equation}
which completes the proof by combining \eqref{vpd} and \eqref{dbc}.
\end{proof}

\subsection{Reconstruction method}

Assume that the surface has a parametric equation:
\[
\partial D=\{\boldsymbol{r}(\theta,\varphi)=(r_1(\theta, \varphi),
r_2(\theta,\varphi), r_3(\theta,\varphi))^\top,~\theta\in(0, \pi),\,
 \varphi\in(0, 2\pi)\},
\]
where $r_j$ are biperiodic functions of $(\theta, \varphi)$ and have 
the Fourier series expansions:
\[
r_j(\theta,\phi)=\sum_{n=0}^{\infty}\sum_{m=-n}^{n} a_{jn}^m {\rm
Re}Y_n^m(\theta,\varphi)+ b_{jn}^m {\rm Im}Y_n^m(\theta,\varphi),
\]
where $Y_n^m$ are the spherical harmonics of order $n$. It suffices to
determine $a_{jn}^m, b_{jn}^m$ in order to reconstruct the
surface. In practice, a cut-off approximation is needed:
\[
 r_{j,N}(\theta, \varphi)=\sum_{n=0}^{N}\sum_{m=-n}^{n} a_{jn}^m {\rm
Re}Y_n^m(\theta,\varphi)+ b_{jn}^m {\rm Im}Y_n^m(\theta,\varphi).
\]

Denote by $D_N$ the approximated obstacle with boundary $\partial D_N$, which
has the parametric equation
\[
\partial D_N=\{\boldsymbol{r}_N(\theta,\varphi)=(r_{1,N}(\theta,\varphi),r_{2,N}
(\theta , \varphi),r_{3,N}(\theta,\varphi))^\top,~\theta\in(0,\pi),\,
 \phi\in(0,2\pi)\}.
\]
Let $\Omega_N=B_R\setminus\bar{D}_N$ and 
\[
 \boldsymbol a_j=(a_{j0}^0, \cdots, a_{jn}^m, \cdots, a_{jN}^N),\quad
\boldsymbol b_j=(b_{j0}^0, \cdots, b_{jn}^m, \cdots, b_{jN}^N),
\]
where $n=0, 1, \dots, N, ~ m=-n, \dots, n.$ Denote the vector of Fourier
coefficients
\[
\boldsymbol{C}=(\boldsymbol a_1, \boldsymbol b_1, \boldsymbol a_2,
\boldsymbol b_2, \boldsymbol a_3, \boldsymbol b_3)^\top=(c_1, c_2, \dots,
c_{6(N+1)^2})^\top\in \mathbb{R}^{6(N+1)^2}
\]
and a vector of scattering data
\[
\boldsymbol{U}=(\boldsymbol{u}(\boldsymbol{x}_1),\dots,\boldsymbol{u}
(\boldsymbol {x}_K))^\top\in \mathbb{C}^{3K},
\]
where $\boldsymbol{x}_k\in\Gamma_R, k=1,\dots,K$. Then the inverse problem can
be formulated to solve an approximate nonlinear equation:
\[
    \mathscr{F}(\boldsymbol{C})=\boldsymbol{U},
\]
where the operator $\mathscr{F}$ maps a vector in $\mathbb{R}^{6(N+1)^2}$ into
a vector in $\mathbb{C}^{3K}$.

\begin{theorem}\label{add}
Let $\boldsymbol{u}_N$ be the solution of the variational problem \eqref{vp}
corresponding to the obstacle $D_N$. The operator $\mathscr{F}$ is
differentiable and its derivatives are given by
\[
\frac{\partial\mathscr{F}_k(\boldsymbol C)}{\partial
c_i}=\boldsymbol{u}'_i(\boldsymbol{x}_k),\quad i=1, \dots, 6(N+1)^2, ~
k=1, \dots, K,
\]
where $\boldsymbol{u}'_i$ is the unique weak solution of the boundary value
problem
\begin{align}\label{aap}
\begin{cases}
\mu\Delta\boldsymbol{u}'_i +(\lambda+\mu)\nabla\nabla\cdot\boldsymbol{u}'_i
+\omega^2\boldsymbol{u}'_i=0 &\quad\text{in}~\Omega_N,\\
\boldsymbol{u}'_i=-q_i \partial_{\boldsymbol\nu_N} \boldsymbol {u}_N
&\quad\text{on} ~\partial D_N.\\
\mathscr{B}\boldsymbol{u}'_i=\mathscr{T}\boldsymbol{u}'_i
&\quad\text{on} ~\Gamma_R.
\end{cases}
\end{align}
Here $\boldsymbol\nu_N=(\nu_{N 1}, \nu_{N 2}, \nu_{N 3})^\top$ is the unit
normal vector on $\partial D_N$ and 
\[
q_i(\theta, \varphi)=
\begin{cases}
\nu_{N 1} {\rm Re}Y_{n}^{m}(\theta, \varphi), & i=n^2+n+m+1,\\
\nu_{N 1} {\rm Im}Y_{n}^{m}(\theta, \varphi), & i=(N+1)^2+n^2+n+m+1,\\
\nu_{N 2} {\rm Re}Y_{n}^{m}(\theta, \varphi), & i=2(N+1)^2+n^2+n+m+1,\\
\nu_{N 2} {\rm Im}Y_{n}^{m}(\theta, \varphi), & i=3(N+1)^2+n^2+n+m+1,\\
\nu_{N 3} {\rm Re}Y_{n}^{m}(\theta, \varphi), & i=4(N+1)^2+n^2+n+m+1,\\
\nu_{N 3} {\rm Im}Y_{n}^{m}(\theta, \varphi), & i=5(N+1)^2+n^2+n+m+1,
\end{cases}
\]
where $n=0, 1, \dots, N, m=-n, \dots, n$.
\end{theorem}

\begin{proof}
 Fix $i\in\{1, \dots, 6(N+1)^2\}$ and $k\in\{1, \dots, K\}$, and let $\{
\boldsymbol e_1, \dots, \boldsymbol e_{6(N+1)^2}\}$ be the set of natural basis
vectors in $\mathbb{R}^{6(N+1)^2}$. By definition, we have
\[
\frac{\partial\mathscr{F}_k(\boldsymbol C)}{\partial c_i}=\lim_{h\to
0}\frac{\mathscr{F}_k(\boldsymbol C+h\boldsymbol e_i)-\mathscr{F}_k(
\boldsymbol C)}{h}.
\]
A direct application of Theorem \ref{dd} shows that the above limit exists and
the limit is the unique weak solution of the boundary value problem \eqref{aap}.
\end{proof}

Consider the objective function
\[
f({\boldsymbol C})=\frac{1}{2}\|\mathscr{F}(\boldsymbol C)-\boldsymbol
U\|^2=\frac{1}{2}\sum_{k=1}^K
|\mathscr{F}_k(\boldsymbol C)-\boldsymbol{u}(\boldsymbol{x}_k)|^2.
\]
The inverse problem can be formulated as the minimization problem:
\[
 \min_{\boldsymbol C}f(\boldsymbol C),\quad
{\boldsymbol C}\in\mathbb{R}^{6(N+1)^2}.
\]
In order to apply the descend method, we have to compute the gradient of the
objective function:
\[
    \nabla f(\boldsymbol C)=\left(\frac{\partial f(\boldsymbol C)}{\partial
c_1}, \dots,    \frac{f(\boldsymbol C)}{\partial c_{6(N+1)^2}}\right)^\top.
\]
We have from Theorem \ref{add} that 
\[
 \frac{\partial f(\boldsymbol C)}{\partial c_i}={ \rm Re}\sum_{k=1}^K
\boldsymbol{u}'_i(\boldsymbol{x}_k)\cdot (\bar{\mathscr{F}}_k(\boldsymbol
C)-\bar{\boldsymbol{u}}(\boldsymbol{x}_k)).
\]

We assume that the scattering data ${\boldsymbol U}$ is available over
a range of frequencies $\omega\in [\omega_{\rm min},\,\omega_{\rm max}]$, which
may be divided into $\omega_{\rm
min}=\omega_0<\omega_1<\cdots<\omega_J=\omega_{\rm max}$. We now propose an
algorithm to reconstruct the Fourier
 coefficients $c_i, i=1,
\dots, 6(N+1)^2$.

\vspace{1ex}
\hrule \hrule 
\vspace{0.8ex} \noindent {\bf Algorithm: Frequency continuation algorithm for
surface reconstruction.}
\vspace{0.8ex} 
\hrule
\vspace{0.8ex} 

\begin{enumerate}

\item  {\bf Initialization}: take an initial guess
$c_{2}=-c_{4}=1.44472 R_0$ and
$c_{3(N+1)^2+2}=c_{3(N+1)^2+4}=1.44472 R_0$,
    $c_{4(N+1)^2+3}=2.0467 R_0$ and $c_i=0$ otherwise. The initial guess is a
ball with radius $R_0$ under the spherical harmonic functions;

 \item {\bf First approximation}: begin with $\omega_0$, let
$k_0=[\omega_0]$, seek an approximation to the functions $r_{j, N}$:
    \[
    r_{j,k_0}=\sum_{n=0}^{k_0}\sum_{m=-n}^{n} a_{jn}^m {\rm Re}
Y_n^m(\theta,\phi)+ b_{jn}^m {\rm Im}Y_n^m(\theta,\phi).
    \]
    Denote $\boldsymbol C^{(1)}_{k_0}=(c_1, c_2, \dots,
c_{6(k_0+1)^2})^\top$ and consider the iteration:
    \begin{equation}\label{descent}
    {\bf C}_{k_0}^{(l+1)}={\bf C}_{k_0}^{(l)}-\tau \nabla f({\bf
    C}_{k_0}^{(l)}),\quad l=1, \dots, L,
    \end{equation}
    where $\tau>0$ and $L>0$ are the step size and the number of iterations
	 for every fixed frequency, respectively.

    \item {\bf Continuation}: increase to $\omega_1$, let $k_1=[\omega_1]$, 
repeat Step 2 with the previous approximation to $r_{j, N}$ as the starting
point. More precisely, approximate $r_{j, N}$ by
    \[
    r_{j, k_1}=\sum_{n=0}^{k_1}\sum_{m=-n}^{n} a_{jn}^m {\rm Re}
Y_n^m(\theta,\phi)+b_{jn}^m {\rm Im} Y_n^m(\theta,\phi),
    \]
    and determine the coefficients $\tilde{c}_i, i=1, \dots, 6(k_1+1)^2$ by
using the descent method starting from the previous result. 

    \item {\bf Iteration}: repeat Step 3 until a prescribed highest frequency
$\omega_J$ is reached.

\end{enumerate}

\vspace{0.8ex}
\hrule\hrule
\vspace{0.8ex} 

\section{Numerical experiments}

In this section, we present two examples to show the effectiveness of the
proposed method. The scattering data is obtained from solving the direct
problem by using the finite element method with the perfectly matched layer
technique, which is implemented via FreeFem++ \cite{h-nm12}. The finite element
solution is interpolated uniformly on $\Gamma_R$. To test the stability, we
add noise to the data:
\[
 \boldsymbol{u}^\delta
(\boldsymbol{x}_k)=\boldsymbol{u}(\boldsymbol{x}_k)(1+\delta\,{\rm rand}),
\quad k=1,\dots, K,
\]
where rand are uniformly distributed random numbers in $[-1,\,1]$ and $\delta$
is the relative noise level, $\boldsymbol{x}_k$ are data points. In our 
experiments, we pick 100 uniformly distributed points $\boldsymbol{x}_k$ on 
$\Gamma_R$, i.e., $K=100$.

In the following two examples, we take $\lambda=2,
\mu=1$, $R=1$. The radius of the initial $R_0=0.5$. The noise level
$\delta=5\%$. The step size in \eqref{descent} is
$\tau=0.005/k_i$ where $k_i=[\omega_i]$. The incident field is taken as a plane
compressional wave.

{\bf Example 1}. Consider a bean-shaped obstacle:
\[
 \boldsymbol{r}(\theta, \varphi)=(r_1(\theta, \varphi), r_2(\theta, \varphi),
r_3(\theta, \varphi))^\top,~\theta\in[0, \pi],\,\varphi\in[0, 2\pi],
\]
where
\begin{align*}
r_1(\theta,
\varphi)&=0.75\left((1-0.05\cos(\pi\cos\theta))\sin\theta\cos\varphi\right)^{1/2}
,\\
r_2(\theta, \varphi)&=0.75\left(
(1-0.005\cos(\pi\cos\theta))\sin\theta\sin\varphi+0.35\cos(\pi\cos\theta)\right)^
{1/2},\\
r_3(\theta, \varphi) &= 0.75\cos\theta.
\end{align*}
The exact surface is plotted in Figure \ref{ex1}(a). This obstacle is
non-convex and is usually difficult to reconstruct the concave part of the
obstacle. The obstacle is illuminated by the compressional wave sent from a
single direction $\boldsymbol d=(0,1,0)^\top$; the frequency
ranges from $\omega_{\rm min}=1$ to $\omega_{\rm max}=5$ with increment 1 at
each continuation step, i.e., $\omega_i=i+1, i=0, \dots, 4$; for any fixed
frequency, repeat $L=100$ times with previous result as starting points. The
step size for the decent method is $0.005/\omega_i$. The number of recovered
coefficients is $6(\omega_i+2)^2$ for corresponding frequency. Figure
\ref{ex1}(b) shows the initial guess which is the ball with radius $R_0=0.5$;
Figure \ref{ex1}(c) shows the final reconstructed surface; Figures
\ref{ex1}(d)--(f) show the cross section of the exact surface along the plane
$x_1=0, x_2=0, x_3=0$, respectively; Figures \ref{ex1}(g)--(i) show the
corresponding cross section for the reconstructed surface along the plane
$x_1=0, x_2=0, x_3=0$, respectively. As is seen, the algorithm effectively
reconstructs the bean-shaped obstacle.

\begin{figure}
\centering
\includegraphics[width=0.325\textwidth]{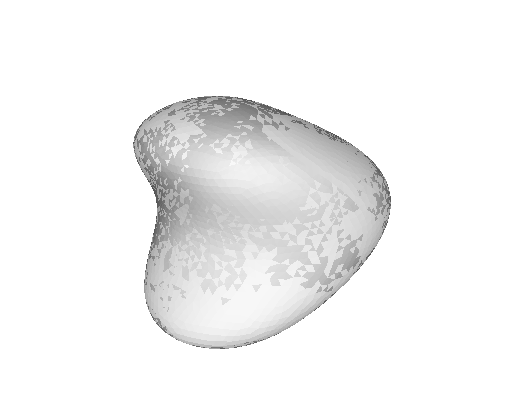}
\includegraphics[width=0.325\textwidth]{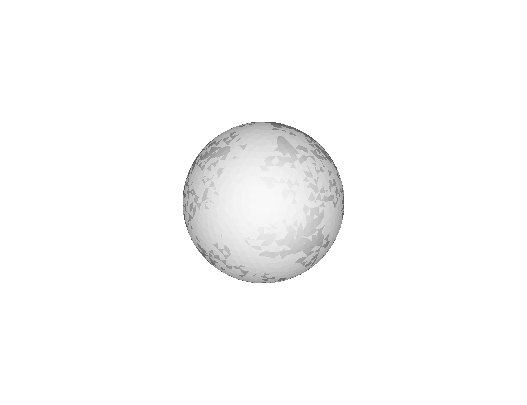}
\includegraphics[width=0.325\textwidth]{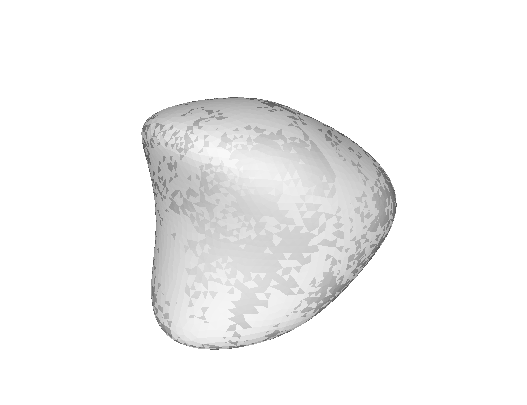}\\
\hspace{-0.5cm} (a) \hspace{4cm} (b) \hspace{3.5cm} (c) \\
\includegraphics[width=0.325\textwidth]{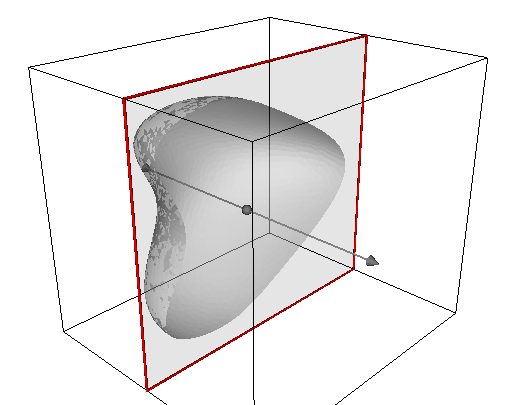}
\includegraphics[width=0.325\textwidth]{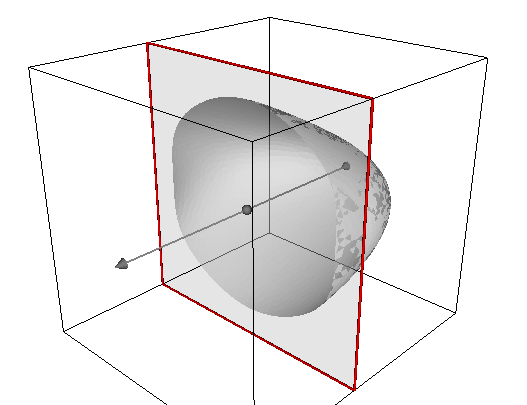}
\includegraphics[width=0.325\textwidth]{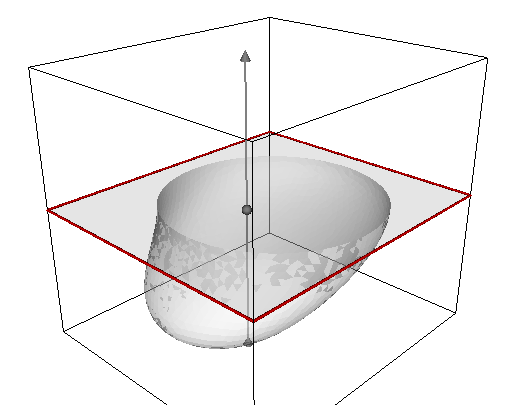}\\
\hspace{-0.5cm} (d) \hspace{4cm} (e) \hspace{3.5cm} (f) \\
\includegraphics[width=0.325\textwidth]{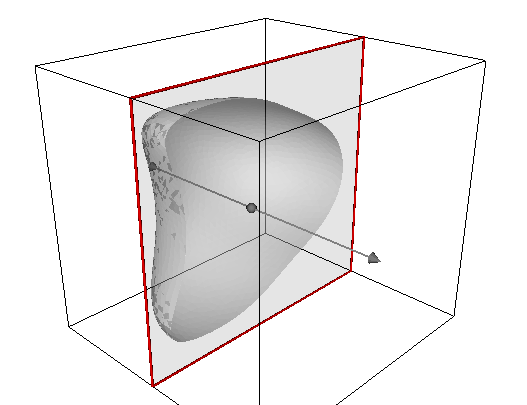}
\includegraphics[width=0.325\textwidth]{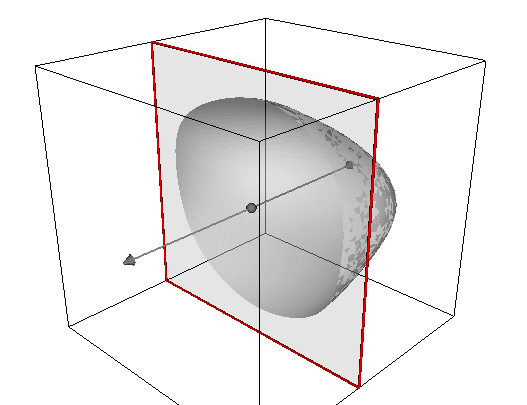}
\includegraphics[width=0.325\textwidth]{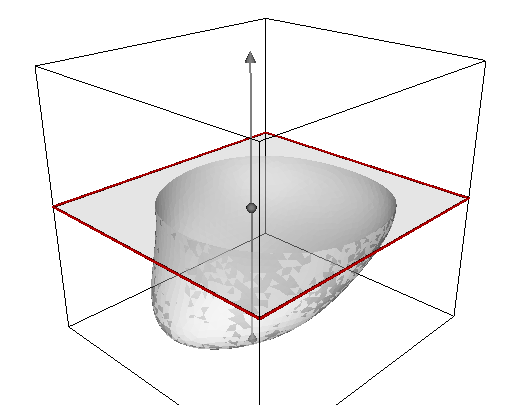}\\
\hspace{-0.5cm} (g) \hspace{4cm} (h) \hspace{3.5cm} (i) \\
\caption{Example 1: A bean-shaped obstacle. (a) the exact surface; (b)
the initial guess; (c) the reconstructed surface; 
(d)--(f) the corresponding cross section of the exact surface along plane
$x_1=0, x_2=0, x_3=0$, respectively; (g)--(i) the corresponding cross 
section of the reconstructed surface along plane $x_1=0, x_2=0, x_3=0$, 
respectively.}
\label{ex1}
\end{figure}

{\bf Example 2}. Consider a cushion-shaped obstacle:
\[
 \boldsymbol{r}(\theta, \varphi)=r(\theta, \varphi)(\sin(\theta)\cos(\varphi),
\sin(\theta)\sin(\varphi), \cos(\theta))^\top,~\theta\in[0,
\pi],\,\varphi\in[0, 2\pi],
\]
where
\[
r(\theta,\varphi)=\left(0.75+0.45(\cos(2\varphi)-1)(\cos(4\theta)-1)\right)^{1/2}.
\]
Figure \ref{ex2}(a) shows the exact surface. This example is much more complex
than the bean-shaped obstacle due to its multiple concave parts. Multiple
incident directions are needed in order to obtain a good result. In this
example, the obstacle is illuminated by the compressional wave from 6
directions, which are the unit vectors pointing to the origin from the
face centers of the cube. The multiple frequencies are the same as the first
example, i.e., the frequency ranges from $\omega_{\rm min}=1$ to
$\omega_{\rm max}=5$ with $\omega_i=i+1, i=0, \dots, 4$. For each fixed
frequency and incident direction, repeat $L=50$ times with previous result as
starting points. The step size for the decent method is $0.005/\omega_i$ and
number of recovered coefficients is $6(\omega_i+2)^2$ for corresponding
frequency. Figure \ref{ex2}(b) shows the initial guess ball with radius
$R_0=0.5$; Figure \ref{ex2}(c) shows the final reconstructed surface; Figure
\ref{ex2}(d)--(f) show the cross section of the exact surface along the plane
$x_1=0, x_2=0, x_3=0$, respectively; while Figure \ref{ex2}(g)--(i) show the
corresponding cross section for the reconstructed surface along the plane
$x_1=0, x_2=0, x_3=0$, respectively. It is clear to note that the algorithm can
also reconstruct effectively the more complex cushion-shaped obstacle.

\begin{figure}
\centering
\includegraphics[width=0.32\textwidth]{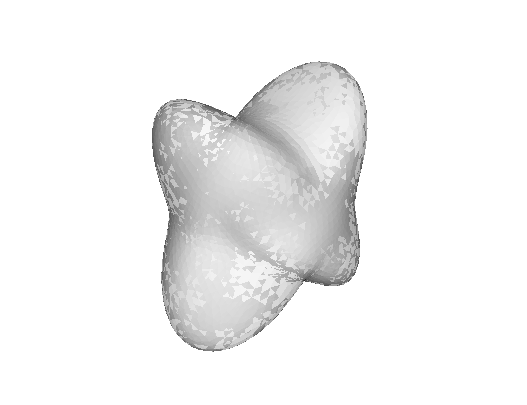}
\includegraphics[width=0.32\textwidth]{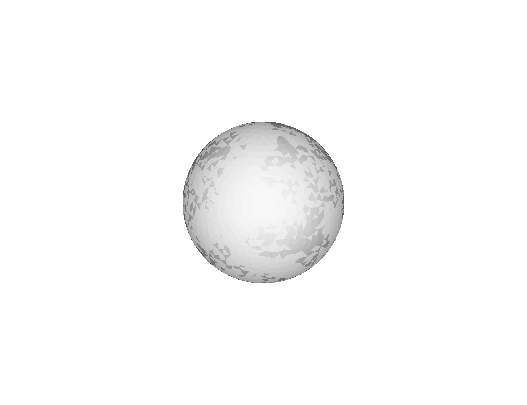}
\includegraphics[width=0.32\textwidth]{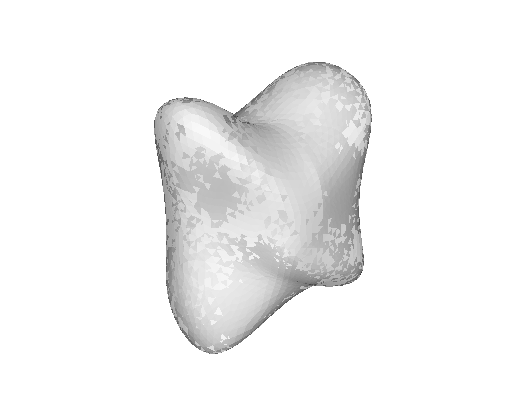}\\
\hspace{-0.5cm} (a) \hspace{4cm} (b) \hspace{3.5cm} (c) \\
\includegraphics[width=0.32\textwidth]{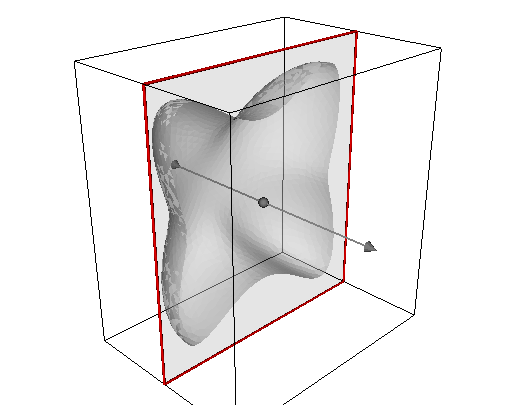}
\includegraphics[width=0.32\textwidth]{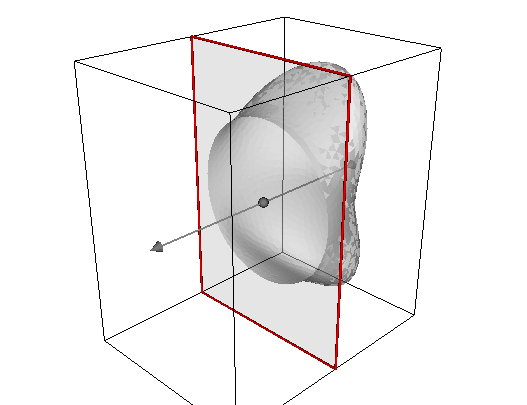}
\includegraphics[width=0.32\textwidth]{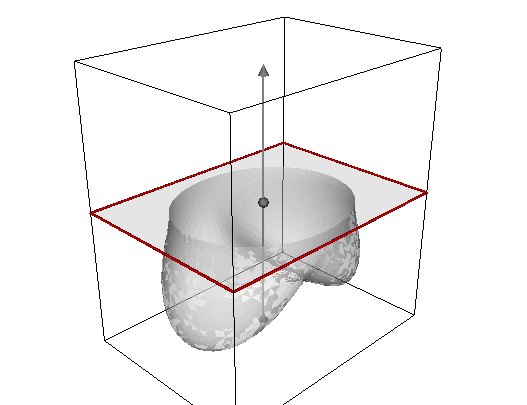}\\
\hspace{-0.5cm} (d) \hspace{4cm} (e) \hspace{3.5cm} (f) \\
\includegraphics[width=0.32\textwidth]{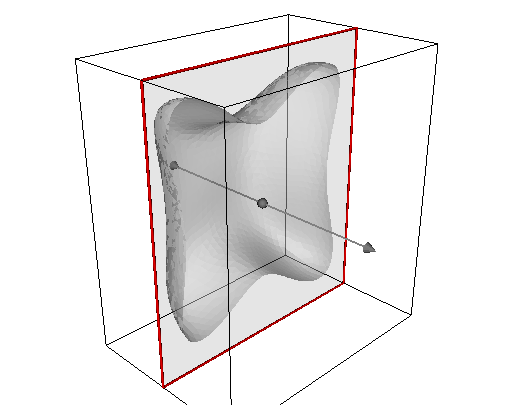}
\includegraphics[width=0.32\textwidth]{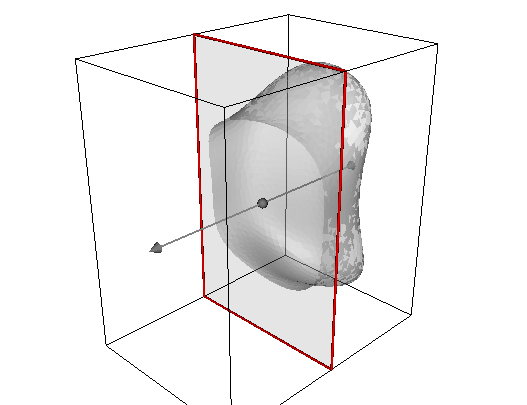}
\includegraphics[width=0.32\textwidth]{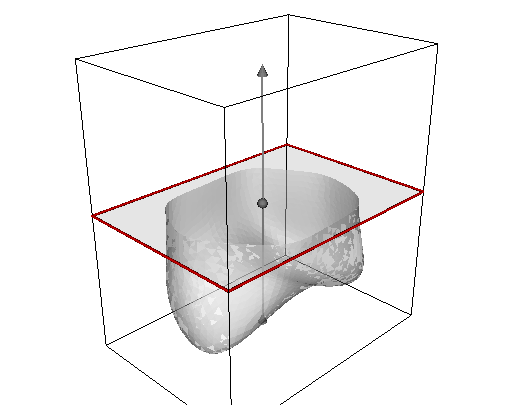}\\
\hspace{-0.5cm} (g) \hspace{4cm} (h) \hspace{3.5cm} (i) \\
\caption{Example 2: A cushion-shaped obstacle. (a) the exact surface; (b)
the initial guess; (c) the reconstructed surface; (d)--(f) the corresponding
cross section of the exact surface along the plane $x_1=0, x_2=0, x_3=0$,
respectively; (d)--(f) the corresponding cross section of the reconstructed
surface along the plane $x_1=0, x_2=0, x_3=0$, respectively.}
\label{ex2}
\end{figure}

\section{Conclusion}

In this paper, we have studied the direct and inverse obstacle scattering
problems for elastic waves in three dimensions. We develop an exact transparent
boundary condition and show that the direct problem has a unique weak solution.
We examine the domain derivative of the total displacement with respect to the
surface of the obstacle. We propose a frequency continuation method for solving
the inverse scattering problem. Numerical examples are presented to
demonstrate the effectiveness of the proposed method. The results show that
the method is stable and accurate to reconstruct surfaces with noise. Future
work includes the surfaces of different boundary conditions and multiple
obstacles where each obstacle's surface has a parametric equation. We hope to be
able to address these issues and report the progress elsewhere in the future.

\appendix

\section{Spherical harmonics and functional spaces}

The spherical coordinates $(r, \theta, \varphi)$ are related to the Cartesian
coordinates $\boldsymbol x=(x_1, x_2, x_3)$ by $x_1=r\sin\theta\cos\varphi,
x_2=r\sin\theta\sin\varphi, x_3=r\cos\theta$. The local orthonormal basis is
\begin{align*}
\boldsymbol e_r &=(\sin\theta\cos\varphi, \sin\theta\sin\varphi,
\cos\theta)^\top, \\
\boldsymbol e_\theta &=(\cos\theta\cos\varphi,
\cos\theta\sin\varphi,-\sin\theta)^\top, \\
\boldsymbol e_\varphi &=(-\sin\varphi, \cos\varphi, 0)^\top.
\end{align*}

Let $\{Y_n^m(\theta, \varphi): n=0, 1, 2, \dots, m=-n, \dots, n\}$ be the
orthonormal sequence of spherical harmonics of order $n$ on the unit sphere.
Define rescaled spherical harmonics
\[
 X_n^m(\theta, \varphi)=\frac{1}{R}Y_n^m(\theta, \varphi). 
\]
It can be shown that $\{X_n^m(\theta, \varphi): n=0, 1, \dots, m=-n,\dots, n\}$
form a complete orthonormal system in $L^2(\Gamma_R)$. 

For a smooth scalar function $u(R, \theta, \varphi)$ defined on $\Gamma_R$, let 
\[
\nabla_{\Gamma_R}u=\partial_\theta u\,\boldsymbol e_\theta
+(\sin\theta)^{-1}\partial_\varphi u\,\boldsymbol e_\varphi
\]
be the tangential gradient on $\Gamma_R$. Define a sequence of vector spherical
harmonics:
\begin{align*}
  \boldsymbol T_n^m(\theta, \varphi)&=\frac{1}{\sqrt{n(n+1)}}\nabla_{\Gamma_R}
X_n^m(\theta, \varphi),\\
\boldsymbol V_n^m(\theta, \varphi)&=\boldsymbol T_n^m(\theta,
\varphi)\times \boldsymbol e_r,\quad
\boldsymbol W_n^m(\theta, \varphi)=X_n^m(\theta, \varphi)\boldsymbol e_r,
\end{align*}
where $n=0, 1, \dots, m=-n, \dots, n$. Using the orthogonality of the vector
spherical harmonics, we can also show that $\{(\boldsymbol T_n^m, \boldsymbol
V_n^m,
\boldsymbol W_n^m): n=0, 1, 2, \dots, m=-n, \dots, n\}$ form a complete
orthonormal system in $\boldsymbol L^2(\Gamma_R)=L^2(\Gamma_R)^3$.

Let $\boldsymbol L^2(\Omega)=L^2(\Omega)^3$ be equipped with the inner product
and norm:
\[
 (\boldsymbol{u}, \boldsymbol{v})=\int_\Omega
    \boldsymbol{u}\cdot\bar{\boldsymbol{v}}\,{\rm d}\boldsymbol{x},\quad
 \|\boldsymbol{u}\|_{\boldsymbol L^2(\Omega)}=(\boldsymbol{u},
\boldsymbol{u})^{1/2}.
\]
Denote by $H^1(\Omega)$ the standard Sobolev space with the
norm given by
\[
 \|u\|_{H^1(\Omega)}=\left(\int_\Omega
|u(\boldsymbol{x})|^2+|\nabla u(\boldsymbol x)|^2\,{\rm
d}\boldsymbol{x}\right)^{1/2}. 
\]
Let $\boldsymbol H^1_{\partial
D}(\Omega)=H^1_{\partial D}(\Omega)^3$, where $H^1_{\partial D}(\Omega):=\{u\in
H^1(\Omega): u=0~\text{on}~\partial D\}$. Introduce the Sobolev space 
\[
 \boldsymbol H({\rm curl}, \Omega)=\{\boldsymbol u\in \boldsymbol L^2(\Omega),
\nabla\times\boldsymbol u\in \boldsymbol L^2(\Omega)\},
\]
which is equipped with the norm
\[
 \|\boldsymbol u\|_{\boldsymbol H({\rm curl},
\Omega)}=\left(\|\boldsymbol u\|^2_{\boldsymbol L^2(\Omega)}
+\|\nabla\times\boldsymbol u\|^2_{\boldsymbol L^2(\Omega)}\right)^{1/2}. 
\]

Denote by $H^s(\Gamma_R)$ the trace functional space which is equipped
with the norm 
\[
 \|u\|_{H^s(\Gamma_R)}=\left(\sum_{n=0}^\infty\sum_{m=-n}^n (1+n(n+1))^s
|u_n^m|^2\right)^{1/2},
\]
where 
\[
u(R, \theta, \varphi)=\sum_{n=0}^\infty\sum_{m=-n}^n u_n^m X_n^m(\theta,
\varphi). 
\]
Let $\boldsymbol H^s(\Gamma_R)=H^s(\Gamma_R)^3$ which is equipped with the
normal
\[
 \|\boldsymbol u\|_{\boldsymbol
H^s(\Gamma_R)}=\left(\sum_{n=0}^\infty\sum_{m=-n}^n (1+n(n+1))^s
|\boldsymbol u_n^m|^2\right)^{1/2},
\]
where $\boldsymbol u_n^m=(u_{1n}^m, u_{2n}^m, u_{3n}^m)^\top$ and 
\[
 \boldsymbol u(R, \theta, \varphi)=\sum_{n=0}^\infty\sum_{m=-n}^n
u_{1n}^m\boldsymbol T_n^m(\theta, \varphi)+u_{2n}^m\boldsymbol V_n^m(\theta,
\varphi)+u_{3n}^m\boldsymbol W_n^m(\theta, \varphi). 
\]
It can be verified that $\boldsymbol H^{-s}(\Gamma_R)$ is the dual space of
$\boldsymbol H^s(\Gamma_R)$ with respect to the inner product
\[
 \langle\boldsymbol{u}, \boldsymbol{v}\rangle_{\Gamma_R}=\int_{\Gamma_R}
\boldsymbol{u}\cdot\bar{\boldsymbol{v}}\,{\rm
d}\gamma=\sum_{n=0}^\infty\sum_{m=-n}^n u_{1n}^m\bar{v}_{1n}^m
+u_{2n}^m\bar{v}_{2n}^m+u_{3n}^m \bar{v}_{3n}^m,
\]
where
\[
 \boldsymbol v(R, \theta, \varphi)=\sum_{n=0}^\infty\sum_{m=-n}^n
v_{1n}^m\boldsymbol T_n^m(\theta, \varphi)+v_{2n}^m\boldsymbol V_n^m(\theta,
\varphi)+v_{3n}^m\boldsymbol W_n^m(\theta, \varphi).
\]

Introduce three tangential trace spaces:
\begin{align*}
 \boldsymbol H_{\rm t}^s(\Gamma_R)&=\{\boldsymbol u\in 
\boldsymbol H^s(\Gamma_R), ~\boldsymbol u\cdot\boldsymbol e_r=0\},\\
\boldsymbol H^{-1/2}({\rm curl}, \Gamma_R)&=\{\boldsymbol u\in\boldsymbol
H^{-1/2}_{\rm t}(\Gamma_R), ~ {\rm curl}_{\Gamma_R}\boldsymbol u\in
H^{-1/2}(\Gamma_R)\},\\
\boldsymbol H^{-1/2}({\rm div}, \Gamma_R)&=\{\boldsymbol u\in \boldsymbol
H^{-1/2}_{\rm t}(\Gamma_R), ~ {\rm div}_{\Gamma_R}\boldsymbol u\in
H^{-1/2}(\Gamma_R)\}.
\end{align*}
For any tangential field $\boldsymbol u\in \boldsymbol H^s_{\rm t}(\Gamma_R)$,
it can be represented in the series expansion
\[
 \boldsymbol u(R, \theta, \varphi)=\sum_{n=0}^\infty\sum_{m=-n}^n u_{1n}^m
\boldsymbol T_n^m(\theta, \varphi)+u_{2n}^m \boldsymbol V_n^m(\theta, \varphi). 
\]
Using the series coefficients, the norm of the space $\boldsymbol H^s_{\rm
t}(\Gamma_R)$ can be characterized by 
\[
 \|\boldsymbol u\|^2_{\boldsymbol H^s_{\rm
t}(\Gamma_R)}=\sum_{n=0}^\infty\sum_{m=-n}^n (1+n(n+1))^s
\left(|u_{1n}^m|^2+|u_{2n}^m|^2 \right);
\]
the norm of the space $\boldsymbol H^{-1/2}({\rm curl}, \Gamma_R)$ can be
characterized by 
\[
 \|\boldsymbol u\|^2_{\boldsymbol H^{-1/2}({\rm curl},
\Gamma_R)}=\sum_{n=0}^\infty\sum_{m=-n}^n
\frac{1}{\sqrt{1+n(n+1)}}|u_{1n}^m|^2+\sqrt{1+n(n+1)}|u_{2n}^m|^2;
\]
the norm of the space $\boldsymbol H^{-1/2}({\rm div}, \Gamma_R)$ can be
characterized by 
\[
 \|\boldsymbol u\|^2_{\boldsymbol H^{-1/2}({\rm div},
\Gamma_R)}=\sum_{n=0}^\infty\sum_{m=-n}^n
\sqrt{1+n(n+1)}|u_{1n}^m|^2+\frac{1}{\sqrt{1+n(n+1)}}|u_{2n}^m|^2. 
\]

Given a vector field $\boldsymbol u$ on $\Gamma_R$, denote by
$ \boldsymbol u_{\Gamma_R}=-\boldsymbol e_r\times(\boldsymbol
e_r\times\boldsymbol u) $ the tangential component of $\boldsymbol u$ on
$\Gamma_R$. Define the inner product in $\mathbb{C}^3$:
$\langle\boldsymbol{u}, \boldsymbol{v}\rangle=\boldsymbol{v}^*\boldsymbol{u},
\forall\, \boldsymbol{u}, \boldsymbol{v}\in\mathbb{C}^3, $
where $\boldsymbol{v}^*$ is the conjugate transpose of $\boldsymbol{v}$.

\section{Transparent boundary conditions}

Recall the Helmholtz decomposition \eqref{hdv}:
\[
 \boldsymbol v=\nabla\phi+\nabla\times\boldsymbol\psi, \quad
\nabla\cdot\boldsymbol\psi=0,
\]
where the scalar potential function $\phi$ satisfies \eqref{he} and
\eqref{ksrc}:
\begin{equation}\label{Aphi}
\begin{cases}
 \Delta\phi+\kappa^2_{\rm p}\phi=0&\quad\text{in}~ \mathbb
R^3\setminus\bar{D},\\
\partial_r\phi-{\rm i}\kappa_{\rm
p}\phi=o(r^{-1})&\quad\text{as}~r\to\infty,
\end{cases}
\end{equation}
the vector potential function $\boldsymbol\psi$ satisfies \eqref{ms} and
\eqref{smrc}:
\begin{equation}\label{Apsi}
 \begin{cases}
  \nabla\times(\nabla\times \boldsymbol\psi)-\kappa^2_{\rm
s}\boldsymbol\psi=0&\quad\text{in}~ \mathbb R^3\setminus\bar{D},\\
(\nabla\times\boldsymbol\psi)\times\hat{\boldsymbol x}-{\rm i}\kappa_{\rm
s}\boldsymbol\psi=o(r^{-1})&\quad\text{as} ~ r\to\infty,
 \end{cases}
\end{equation}
where $r=|\boldsymbol x|$ and $\hat{\boldsymbol x}=\boldsymbol x/r$.

In the exterior domain $\mathbb R^3\setminus \bar{B}_R$, the solution
$\phi$ of \eqref{Aphi} satisfies
\begin{equation}\label{sphi}
 \phi(r, \theta, \varphi)=\sum_{n=0}^\infty\sum_{m=-n}^n
\frac{h^{(1)}_n(\kappa_{\rm p}r)}{h^{(1)}_n(\kappa_{\rm p}R)}\phi_n^m
X_n^m(\theta, \varphi),
\end{equation}
where $h^{(1)}_n$ is the spherical Hankel function of the first kind with order
$n$ and 
\[
 \phi_n^m=\int_{\Gamma_R} \phi(R, \theta, \varphi)\bar{X}_n^m(\theta,
\varphi){\rm d}\gamma.
\]
We define the boundary operator $\mathscr T_1$ such that 
\begin{equation}\label{bphi}
 (\mathscr{T}_1\phi)(R, \theta,
\varphi)=\frac{1}{R}\sum_{n=0}^\infty\sum_{m=-n}^n z_n(\kappa_{\rm p}R)\phi_n^m
X_n^m(\theta, \varphi),
\end{equation}
where $ z_n(t)=th_n^{(1)'}(t)/h_n^{(1)}(t)$ satisfies (cf. \cite[Theorem
2.6.1]{n-00})
\begin{equation}\label{zn}
-(n+1) \leq {\rm Re}z_n(t)\leq -1,\quad 0<{\rm Im}z_n(t)\leq t.
\end{equation}
Evaluating the derivative of \eqref{sphi} with respect to $r$ at $r=R$ and
using \eqref{bphi}, we get the transparent boundary condition for the scalar
potential function $\phi$:
\begin{equation}\label{tphi}
 \partial_r\phi=\mathscr{T}_1\phi\quad\text{on}~ \Gamma_R. 
\end{equation}

The following result can be easily shown from \eqref{bphi}--\eqref{zn}. 
\begin{lemma}\label{bt1}
 The operator $\mathscr T_1$ is bounded from $H^{1/2}(\Gamma_R)$ to
$H^{-1/2}(\Gamma_R)$. Moreover, it satisfies
\[
 {\rm Re}\langle\mathscr T_1 u, u\rangle_{\Gamma_R}\leq 0,\quad {\rm
Im}\langle\mathscr T_1 u, u\rangle_{\Gamma_R}\geq 0,\quad\forall u\in
H^{1/2}(\Gamma_R). 
\]
If ${\rm Re}\langle\mathscr T_1 u, u\rangle_{\Gamma_R}=0$ or ${\rm
Im}\langle\mathscr T_1 u, u\rangle_{\Gamma_R}=0$, then $u=0$ on $\Gamma_R$. 
\end{lemma}

Define an auxiliary function $\boldsymbol\varphi=({\rm i}\kappa_{\rm
s})^{-1}\nabla\times\boldsymbol\psi$. We have from \eqref{Apsi} that 
\begin{equation}\label{me}
  \nabla\times\boldsymbol\psi-{\rm i}\kappa_{\rm
s}\boldsymbol\varphi=0, \quad
\nabla\times\boldsymbol\varphi+{\rm i}\kappa_{\rm
s}\boldsymbol\psi=0,
\end{equation}
which are Maxwell's equations. Hence $\boldsymbol\phi$ and $\boldsymbol\psi$
plays the role of the electric field and the magnetic field, respectively.

Introduce the vector wave functions
\begin{equation}\label{vwf}
 \begin{cases}
  \boldsymbol M_n^m(r, \theta, \varphi)=\nabla\times(\boldsymbol x
h_n^{(1)}(\kappa_{\rm s}r)X_n^m(\theta, \varphi)),\\
\boldsymbol N_n^m(r, \theta, \varphi)=({\rm i}\kappa_{\rm
s})^{-1}\nabla\times\boldsymbol M_n^m(r, \theta, \varphi),
 \end{cases}
\end{equation}
which are the radiation solutions of \eqref{me} in $\mathbb
R^3\setminus\{0\}$ (cf. \cite[Theorem 9.16]{m-03}):
\[
 \nabla\times\boldsymbol M_n^m(r, \theta, \varphi)-{\rm i}\kappa_{\rm
s}\boldsymbol N_n^m(r, \theta, \varphi)=0,\quad \nabla\times\boldsymbol
N_n^m(r, \theta, \varphi)+{\rm i}\kappa_{\rm s}\boldsymbol M_n^m(r, \theta,
\varphi)=0.
\]
Moreover, it can be verified from \eqref{vwf} that they satisfy 
\begin{equation}\label{sm}
 \boldsymbol M_n^m=h_n^{{(1)}}(\kappa_{\rm
s}r)\nabla_{\Gamma_R}X_n^m\times\boldsymbol e_r
\end{equation}
and
\begin{equation}\label{sn}
 \boldsymbol N_n^m=\frac{\sqrt{n(n+1)}}{{\rm i}\kappa_{\rm
s}r}(h_n^{(1)}(\kappa_{\rm s}r)+\kappa_{\rm s}r h_n^{(1)'}(\kappa_{\rm
s}r))\boldsymbol T_n^m+\frac{n(n+1)}{{\rm i}\kappa_{\rm
s}r}h_n^{(1)}(\kappa_{\rm s}r)\boldsymbol W_n^m. 
\end{equation}
In the domain $\mathbb R^3\setminus\bar{B}_R$, the solution of
$\boldsymbol\psi$ in \eqref{me} can be written in the series
\begin{equation}\label{spsi}
  \boldsymbol\psi=\sum_{n=0}^\infty\sum_{m=-n}^n \alpha_n^m\boldsymbol
N_n^m+\beta_n^m\boldsymbol M_n^m,
\end{equation}
which is uniformly convergent on any compact subsets in $\mathbb
R^3\setminus\bar{B}_R$. Correspondingly, the solution of $\boldsymbol\varphi$
in \eqref{me} is given by 
\begin{equation}\label{svpi}
 \boldsymbol\varphi=({\rm
i}\kappa_{\rm s})^{-1}\nabla\times\boldsymbol\psi=\sum_{n=0}^\infty\sum_{m=-n}^n
\beta_n^m\boldsymbol N_n^m-\alpha_n^m\boldsymbol M_n^m. 
\end{equation}

It follows from \eqref{sm}--\eqref{sn} that 
\begin{align*}
 -\boldsymbol e_r\times(\boldsymbol e_r\times \boldsymbol
M_n^m)&=-\sqrt{n(n+1)}h_n^{(1)}(\kappa_{\rm s}r)\boldsymbol V_n^m,\\
-\boldsymbol e_r\times(\boldsymbol e_r\times \boldsymbol
N_n^m)&=\frac{\sqrt{n(n+1)}}{{\rm i}\kappa_{\rm
s}r}(h_n^{(1)}(\kappa_{\rm s}r)+\kappa_{\rm s}r h_n^{(1)'}(\kappa_{\rm
s}r))\boldsymbol T_n^m
\end{align*}
and
\begin{align*}
 \boldsymbol e_r\times \boldsymbol M_n^m&=\sqrt{n(n+1)}h_n^{(1)}(\kappa_{\rm
s}r)\boldsymbol T_n^m,\\
\boldsymbol e_r\times \boldsymbol N_n^m&=\frac{\sqrt{n(n+1)}}{{\rm
i}\kappa_{\rm s}r}(h_n^{(1)}(\kappa_{\rm s}r)+\kappa_{\rm s}r
h_n^{(1)'}(\kappa_{\rm s}r))\boldsymbol V_n^m.
\end{align*}
Therefore, by \eqref{spsi}, the tangential component of $\boldsymbol\psi$ on
$\Gamma_R$ is 
\[
 \boldsymbol\psi_{\Gamma_R}=\sum_{n=0}^\infty\sum_{m=-n}^n
\frac{\sqrt{n(n+1)}}{{\rm i}\kappa_{\rm s}R}(h_n^{(1)}(\kappa_{\rm
s}R)+\kappa_{\rm s}R h_n^{(1)'}(\kappa_{\rm s}R))\alpha_n^m \boldsymbol
T_n^m +\sqrt{n(n+1)}h_n^{(1)}(\kappa_{\rm s}R)\beta_n^m \boldsymbol V_n^m.
\]
Similarly, by \eqref{svpi}, the tangential trace of $\boldsymbol\varphi$ on
$\Gamma_R$ is
\begin{align*}
 \boldsymbol\varphi\times\boldsymbol e_r &=\sum_{n=0}^\infty\sum_{m=-n}^n
\sqrt{n(n+1)}h_n^{(1)}(\kappa_{\rm s}R)\alpha_n^m \boldsymbol T_n^m\\
&\qquad-\frac{\sqrt{n(n+1)}}{{\rm i}\kappa_{\rm s}R}(h_n^{(1)}(\kappa_{\rm
s}R)+\kappa_{\rm s}R h_n^{(1)'}(\kappa_{\rm s}R))\beta_n^m \boldsymbol
V_n^m.
\end{align*}

Given any tangential component of the electric field on
$\Gamma_R$ with the expression
\[
 \boldsymbol u=\sum_{n=0}^\infty\sum_{m=-n}^n u_{1n}^m \boldsymbol T_n^m
+u_{2n}^m\boldsymbol V_n^m,
\]
we define 
\begin{equation}\label{bpsi}
 \mathscr{T}_2\boldsymbol u=\sum_{n=0}^\infty\sum_{m=-n}^n
\frac{{\rm i}\kappa_{\rm s}R}{1+z_n(\kappa_{\rm s}R)}u_{1n}^m \boldsymbol T_n^m
+\frac{1+z_n(\kappa_{\rm s}R)}{{\rm i}\kappa_{\rm s}R} u_{2n}^m \boldsymbol
V_n^m.
\end{equation}
Using \eqref{bpsi}, we obtain the transparent boundary condition
for $\boldsymbol\psi$:
\begin{equation}\label{tpsi}
 (\nabla\times\boldsymbol\psi)\times\boldsymbol e_r={\rm i}\kappa_{\rm
s}\mathscr{T}_2\boldsymbol\psi_{\Gamma_R}\quad\text{on}~ \Gamma_R. 
\end{equation}

The following result can also be easily shown from \eqref{zn} and \eqref{bpsi}
\begin{lemma}\label{bt2}
 The operator $\mathscr T_2$ is bounded from $\boldsymbol H^{1/2}({\rm curl},
\Gamma_R)$ to $\boldsymbol H^{-1/2}({\rm div}, \Gamma_R)$. Moreover, it
satisfies
\[
 {\rm Re}\langle\mathscr T_2 \boldsymbol u, \boldsymbol u\rangle_{\Gamma_R}\geq
0,\quad\forall \boldsymbol u\in \boldsymbol H^{1/2}({\rm curl}, \Gamma_R). 
\]
If ${\rm Re}\langle\mathscr T_2 \boldsymbol u, \boldsymbol
u\rangle_{\Gamma_R}=0$, then $\boldsymbol u=0$ on $\Gamma_R$. 
\end{lemma}

\section{Fourier coefficients}

We derive the mutual representations of the Fourier coefficients between
$\boldsymbol v$ and $(\phi, \boldsymbol\psi)$. First we have from \eqref{sphi}
that
\begin{equation}\label{D-phi}
  \phi(r, \theta, \varphi)=\sum_{n=0}^\infty\sum_{m=-n}^n
\frac{h^{(1)}_n(\kappa_{\rm p}r)}{h^{(1)}_n(\kappa_{\rm p}R)}\phi_n^m
X_n^m(\theta, \varphi).
\end{equation}
Substituting \eqref{sm}--\eqref{sn} into \eqref{spsi} yields 
\begin{align}\label{D-psi1}
 \boldsymbol\psi(r, \theta,
\varphi)&=\sum_{n=0}^\infty\sum_{m=-n}^n\frac{\sqrt{n(n+1)}}{{\rm i}\kappa_{\rm
s}r}(h_n^{(1)}(\kappa_{\rm s}r)+\kappa_{\rm s}r h_n^{(1)'}(\kappa_{\rm
s}r))\alpha_n^m \boldsymbol T_n^m\notag\\
&\qquad+\sqrt{n(n+1)}h_n^{(1)}(\kappa_{\rm s}r)\beta_n^m \boldsymbol
V_n^m+\frac{n(n+1)}{{\rm i}\kappa_{\rm s}r}h_n^{(1)}(\kappa_{\rm
s}r)\alpha_n^m \boldsymbol W_n^m. 
\end{align}
Given $\boldsymbol\psi$ on $\Gamma_R$, it has the Fourier expansion:
\begin{equation}\label{D-psi2}
 \boldsymbol\psi(R, \theta, \varphi)=\sum_{n=0}^\infty\sum_{m=-n}^n \psi_{1n}^m
\boldsymbol T_n^m(\theta, \varphi)+\psi_{2n}^m\boldsymbol V_n^m(\theta,
\varphi)+\psi_{3n}^m\boldsymbol W_n^m(\theta, \varphi).
\end{equation}
Evaluating \eqref{D-psi1} at $r=R$ and then comparing it with
\eqref{D-psi2}, we get
\begin{equation}\label{ab}
 \alpha_n^m=\frac{{\rm i}\kappa_{\rm s}R
}{n(n+1)h_n^{(1)}(\kappa_{\rm s}R)}\psi_{3n}^m,\quad
\beta_n^m=\frac{1}{\sqrt{n(n+1)}h_n^{(1)}(\kappa_{\rm s}R)}\psi_{2n}^m.
\end{equation}
Plugging \eqref{ab} back into \eqref{D-psi1} gives 
\begin{align}\label{D-psi}
  \boldsymbol\psi(r, \theta, \varphi)&=\sum_{n=0}^\infty\sum_{m=-n}^n
\left(\frac{R}{r}\right)\left(\frac{h_n^{(1)}(\kappa_{\rm s}r)+\kappa_{\rm s}r
h_n^{(1)'}(\kappa_{\rm s}r)}{\sqrt{n(n+1)}h_n^{(1)}(\kappa_{\rm
s}R)}\right)\psi_{3n}^m\boldsymbol T_n^m\notag\\
&\qquad+\left(\frac{h_n^{(1)}(\kappa_{\rm s}r)}{h_n^{(1)}(\kappa_{\rm
s}R)}\right)\psi_{2n}^m\boldsymbol
V_n^m+\left(\frac{R}{r}\right)\left(\frac{h_n^ {(1)} 
(\kappa_{\rm  s}r)}{h_n^{(1)}(\kappa_{\rm s} R)}\right)\psi_{3n}^m\boldsymbol
W_n^m.     
\end{align}

Noting $\nabla\phi=\partial_r\phi\, \boldsymbol e_r
+\frac{1}{r}\nabla_{\Gamma_R}\phi$, we have from \eqref{D-phi} and \eqref{D-psi}
that
\begin{align*}
 \nabla\phi &=\sum_{n=0}^\infty\sum_{m=-n}^n \left(\frac{\kappa_{\rm
p}h^{(1)'}_n(\kappa_{\rm p}r)}{h^{(1)}_n(\kappa_{\rm p}R)}\right)\phi_n^m X_n^m
\boldsymbol e_r +\left(\frac{h^{(1)}_n(\kappa_{\rm
p}r)}{r h^{(1)}_n(\kappa_{\rm p}R)}\right)\phi_n^m \nabla_{\Gamma_R}X_n^m\\
&=\sum_{n=0}^\infty\sum_{m=-n}^n \left(\frac{\kappa_{\rm
p}h^{(1)'}_n(\kappa_{\rm p}r)}{h^{(1)}_n(\kappa_{\rm
p}R)}\right)\phi_n^m\boldsymbol W_n^m+\left(\frac{\sqrt{n(n+1)} h^{(1)}
_n(\kappa_{\rm p}r)}{r h^{(1)}_n(\kappa_{\rm p}R)}\right)\phi_n^m\boldsymbol
T_n^m. 
\end{align*}
and
\[
 \nabla\times\boldsymbol\psi=\sum_{n=0}^\infty\sum_{m=-n}^n \boldsymbol I_{1n}^m
+ \boldsymbol I_{2n}^m + \boldsymbol I_{3n}^m,
\]
where
\begin{align*}
 \boldsymbol I_{1n}^m&=\nabla\times\left[
\left(\frac{R}{r}\right)\left(\frac{h_n^{(1)}(\kappa_{\rm s}r)+\kappa_{\rm s}r
h_n^{(1)'}(\kappa_{\rm s}r)}{\sqrt{n(n+1)}h_n^{(1)}(\kappa_{\rm
s}R)}\right)\psi_{3n}^m\boldsymbol T_n^m\right]\\
&=\frac{R h_n^{(1)}(\kappa_{\rm s}r)}{\sqrt{n(n+1)}h_n^{(1)}(\kappa_{\rm
s}R)}\left(\kappa^2_{\rm s}-\frac{n(n+1)}{r^2}\right)\psi_{3n}^m\boldsymbol
V_n^m,\\
\boldsymbol I_{2n}^m &=\nabla\times\left[\left(\frac{h_n^{(1)}(\kappa_{\rm
s}r)}{h_n^{(1)}(\kappa_{\rm s}R)}\right)\psi_{2n}^m\boldsymbol V_n^m\right]\\
&=\left(\frac{h_n^{(1)}(\kappa_{\rm s}r)+\kappa_{\rm s}r
h_n^{(1)'}(\kappa_{\rm s}r)}{r h_n^{(1)}(\kappa_{\rm s}R)}\right)
\psi_{2n}^m\boldsymbol T_n^m +\frac{\sqrt{n(n+1)}h_n^{(1)}(\kappa_{\rm s}r)}{r
h_n^{(1)}(\kappa_{\rm s}R)}\psi_{2n}^m \boldsymbol W_n^m,\\
\boldsymbol I_{3n}^m
&=\nabla\times\left[\left(\frac{R}{r}\right)\left(\frac{h_n^{(1)} 
(\kappa_{\rm  s}r)}{h_n^{(1)}(\kappa_{\rm s} R)}\right)\psi_{3n}^m\boldsymbol
W_n^m\right]
=\frac{R\sqrt{n(n+1)} h_n^{(1)}(\kappa_{\rm s}r)}{r^2 h_n^{(1)}(\kappa_{\rm
s}R)}\psi_{3n}^m\boldsymbol V_n^m. 
\end{align*}
Combining the above equations and noting $\boldsymbol
v=\nabla\phi+\nabla\times\boldsymbol\psi$, we obtain 
\begin{align}\label{D-v}
& \boldsymbol v(r, \theta, \varphi)=\sum_{n=0}^\infty\sum_{m=-n}^n
\left(\frac{\sqrt{n(n+1)} h^{(1)}_n(\kappa_{\rm
p}r)}{r h^{(1)}_n(\kappa_{\rm p}R)}\phi_n^m+\frac{(h_n^{(1)}(\kappa_{\rm
s}r)+\kappa_{\rm s}r h_n^{(1)'}(\kappa_{\rm s}r))}{r
h_n^{(1)}(\kappa_{\rm s}R)}\psi_{2n}^m\right) \boldsymbol T_n^m\notag\\
&\qquad+\frac{\kappa^2_{\rm s}R h_n^{(1)}(\kappa_{\rm
s}r) }{\sqrt{n(n+1)}h_n^{(1)}(\kappa_{\rm
s}R)}\psi_{3n}^m\boldsymbol V_n^m +\left(\frac{\kappa_{\rm
p}h^{(1)'}_n(\kappa_{\rm p}r)}{h^{(1)}_n(\kappa_{\rm
p}R)}\phi_n^m+\frac{\sqrt{n(n+1)}h_n^{(1)}(\kappa_{\rm s}r)}{r
h_n^{(1)}(\kappa_{\rm s}R)}\psi_{2n}^m\right) \boldsymbol W_n^m,
\end{align}
which gives
\begin{align}\label{D-v1}
  \boldsymbol v(R, \theta,
\varphi)=&\sum_{n=0}^\infty\sum_{m=-n}^n\frac{1}{R}\left(\sqrt{n(n+1)}\phi_n^m
+(1+z_n(\kappa_{\rm s}R))\psi_{2n}^m\right)\boldsymbol T_n^m\notag\\
&+\frac{\kappa^2_{\rm s} R }{\sqrt{n(n+1)}}\psi_{3n}^m\boldsymbol V_n^m
+\frac{1}{R}\left(z_n(\kappa_{\rm p}R)\phi_n^m
+\sqrt{n(n+1)}\psi_{2n}^m\right) \boldsymbol W_n^m.
\end{align}

On the other hand, $\boldsymbol v$ has the Fourier expansion:
\begin{equation}\label{D-v2}
 \boldsymbol v(R, \theta, \varphi)=\sum_{n=0}^\infty\sum_{m=-n}^n v_{1n}^m
\boldsymbol T_n^m+v_{2n}^m\boldsymbol V_n^m+v_{3n}^m\boldsymbol W_n^m. 
\end{equation}
Comparing \eqref{D-v1} with \eqref{D-v2}, we obtain 
\begin{equation}\label{vtp}
\begin{cases}
 v_{1n}^m =\dfrac{\sqrt{n(n+1)}}{R}\phi_n^m
+\dfrac{(1+z_n(\kappa_{\rm s}R))}{R}\psi_{2n}^m,\\[5pt]
v_{2n}^m =\dfrac{ \kappa^2_{\rm
s}R}{\sqrt{n(n+1)}}\psi_{3n}^m,\\[5pt]
v_{3n}^m =\dfrac{z_n(\kappa_{\rm p}R)}{R}\phi_n^m
+\dfrac{\sqrt{n(n+1)}}{R}\psi_{2n}^m,
\end{cases}
\end{equation}
and
\begin{equation}\label{ptv}
\begin{cases}
 \phi_n^m =\dfrac{R(1+z_n(\kappa_{\rm
s}R))}{\Lambda_n}v_{3n}^m-\dfrac{R\sqrt{n(n+1)}}{\Lambda_n}v_{1n}^m,\\[5pt]
 \psi_{2n}^m =\dfrac{R z_n(\kappa_{\rm
p}R)}{\Lambda_n}v_{1n}^m-\dfrac{R\sqrt{n(n+1)}}{\Lambda_n}v_{3n}^m,\\[5pt]
 \psi_{3n}^m =\dfrac{\sqrt{n(n+1)}}{\kappa_{\rm s}^2 R}v_{2n}^m,
 \end{cases}
\end{equation}
where 
\[
 \Lambda_n=z_n(\kappa_{\rm p}R)(1+z_n(\kappa_{\rm s}R))-n(n+1).
\]
Noting \eqref{zn}, we have from a simple calculation that 
\[
 {\rm Im}\Lambda_n={\rm Re}z_n(\kappa_{\rm p}R){\rm Im}z_n(\kappa_{\rm
s}R)+(1+{\rm Re}z_n(\kappa_{\rm s}R)){\rm Im}z_n(\kappa_{\rm p}R)<0,
\]
which implies that $\Lambda_n\neq 0$ for $n=0, 1, \dots.$

\end{document}